\documentclass[11pt]{amsart}
\usepackage{a4wide}
\usepackage{amsmath}
\usepackage[utf8]{inputenc}
\usepackage{amssymb}
\usepackage{amsopn}
\usepackage{epsfig}
\usepackage{amsfonts}
\usepackage{latexsym}
\usepackage{enumitem}
\usepackage{graphicx}
\usepackage{calrsfs}
\usepackage{tikz}

 \usetikzlibrary{arrows,automata}
\DeclareMathAlphabet{\pazocal}{OMS}{zplm}{m}{n}

\newtheorem{theorem}{Theorem}[section]
\newtheorem{lemma}[theorem]{Lemma}
\newtheorem{proposition}[theorem]{Proposition}

\theoremstyle{definition}

\theoremstyle{remark}
\newtheorem{remark}[theorem]{Remark}

\numberwithin{equation}{section}

\newcommand{\N}{\ensuremath{\mathbb{N}}}

\renewcommand{\c}{ {\mathbf{c}}}
\renewcommand{\d}{ {\mathbf{d}}}
\newcommand{\us}{\mathbf{U}}
\newcommand{\ub}{\mathcal{U}}

\newcommand{\uu}{\pazocal{U}}
\newcommand{\vb} {\mathcal{V}}
\newcommand{\set}[1]{\left\{#1\right\}}

\newcommand{\ep}{\varepsilon}
\newcommand{\f}{\infty}
\newcommand{\de}{\delta}

\newcommand{\al}{\alpha}

\newcommand{\si}{\sigma}

\newcommand{\ra}{\rightarrow}
\newcommand{\lle}{\preccurlyeq}
\newcommand{\lge}{\succcurlyeq}

\begin{document}

\title[Univoque bases of real numbers]{Univoque bases of real numbers:  local dimension, Devil's staircase and isolated points}

\author[D. Kong]{Derong Kong}
\address[D. Kong]{College  of Mathematics and Statistics, Chongqing University, Chongqing 401331, People's Republic of China.}
\email{derongkong@126.com}

\author[W. Li]{Wenxia Li}
\address[W. Li]{Department of Mathematics, Shanghai Key Laboratory of PMMP, East China Normal University, Shanghai 200062,
People's Republic of China}
\email{wxli@math.ecnu.edu.cn}

\author[F. L\"u]{Fan L\"{u} }
\address[F. L\"u]{Department of Mathematics, Sichuan Normal University, Chengdu 610068, People's Republic of China}
\email{lvfan1123@163.com}

\author[Z. Wang]{Zhiqiang Wang}
\address[Z. Wang]{Department of Mathematics, Shanghai Key Laboratory of PMMP, East China Normal University, Shanghai 200062,
People's Republic of China}
\email{zhiqiangwzy@163.com}

\author[J. Xu]{Jiayi Xu}
\address[J. Xu]{Department of Mathematics, Shanghai Key Laboratory of PMMP, East China Normal University, Shanghai 200062,
People's Republic of China}
\email{dkxujy@163.com}

\dedicatory{}



\begin{abstract}
Given a positive integer $M$ and a real number $x>0$, let $\ub(x)$ be the set of all bases $q\in(1, M+1]$ for which there exists a unique sequence $(d_i)=d_1d_2\ldots$ with each digit $d_i\in\set{0,1,\ldots, M}$ satisfying
\[
x=\sum_{i=1}^\f\frac{d_i}{q^i}.
\]
The sequence $(d_i)$ is called a \emph{$q$-expansion} of $x$.
In this paper we investigate the local dimension of $\ub(x)$ and prove a `variation principle' for unique non-integer base expansions. We also determine the critical values of $\ub(x)$  such that when $x$ passes the first critical value the set $\ub(x)$ changes from a set with positive Hausdorff dimension  to  a countable set, and when $x$ passes the second critical value the set $\ub(x)$ changes from an infinite set to a singleton. Denote by $\us(x)$ the set of all unique $q$-expansions of $x$ for  $q\in\ub(x)$. We give the Hausdorff dimension of $\us(x)$ and show that the dimensional function $x\mapsto\dim_H\us(x)$ is a non-increasing Devil's staircase. Finally, we investigate the topological structure of $\ub(x)$. Although the set  $\ub(1)$ has no isolated points, we prove that for typical $x>0$ the set $\ub(x)$ contains isolated points.
\end{abstract}
\maketitle

\section{Introduction}\label{sec: Introduction}

Given a positive integer $M$ and a real number $q\in(1, M+1]$, each point $x\in[0, M/(q-1)]$ can be written as
\begin{equation}\label{eq:pi-q}
x=\pi_q((d_i)):=\sum_{i=1}^\f\frac{d_i}{q^i},\quad d_i\in\set{0, 1,\ldots, M}~\forall i\ge 1.
\end{equation}
The infinite sequence $(d_i)=d_1d_2\cdots$   is called a $q$-\emph{expansion} of $x$ with respect to the \emph{alphabet} $\{0,1,\cdots,M\}$.

Expansions in non-integer bases were pioneered by R\'{e}nyi \cite{Renyi_1957} and Parry \cite{Parry_1960}. Different from the integer base expansions Sidorov \cite{Sidorov_2003} (see also \cite{Dajani_DeVries_2007}) showed that for any $q\in(1, M+1)$ Lebesgue almost every $x\in[0, M/(q-1)]$ has  a continuum of $q$-expansions.  Furthermore, Erd\H os et al.~\cite{Erdos_Joo_Komornik_1990, Erdos_Horvath_Joo_1991, Erdos_Joo_1992} showed that for any  $k\in\N\cup\set{\aleph_0}$ there exist $q\in(1, M+1)$ and $x\in[0, M/(q-1)]$ such that $x$ has precisely  $k$ different  $q$-expansions  (see also cf.~\cite{  Sidorov_2009}). In particular, there is a great interest in unique $q$-expansions due to {their} close connections with open dynamical systems (cf.~\cite{DeVries_Komornik_2008, Glendinning_Sidorov_2001, Komornik-Kong-Li-17}). For more information on expansions in  non-integer bases we refer to the surveys \cite{ Komornik_2011, Sidorov_2003_survey} and the survey chapter \cite{deVries-Komornik-2016}.

For $q>1$ let $\uu_q$ be  the \emph{univoque set} of $x\in I_q:=[0,M/(q-1)]$ having a unique $q$-expansion, and let $\us_q:=\pi_q^{-1}(\uu_q)$ be the set of corresponding $q$-expansions.
Dual to the univoque set $\uu_q$  we consider in this paper the set of \emph{univoque bases} of real numbers.
For $x\ge 0$  let $\ub(x)$ be the set of bases $q\in(1, M+1]$ such that $x$ has a unique $q$-expansion, i.e.,
 \[
 \ub(x)=\set{q\in(1, M+1]: x\in\uu_q}.
 \]
Clearly, for $x=0$ the set $\ub(0)=(1,M+1]$, because  for each $q\in(1, M+1]$ the point  $0$ always has a unique $q$-expansion $0^\f=00\cdots$.  So, it is interesting to investigate the set $\ub(x)$ for $x>0$.

When $x=1$, the set $\ub=\ub(1)$  is well understood. Erd\H os et al.~\cite{Erdos_Joo_Komornik_1990} showed  that   $\ub$ is a  Lebesgue null set of first category but it is uncountable. Later Dar\'oczy and K\'atai \cite{Darczy_Katai_1995} showed that $\ub$ has full Hausdorff dimension. Clearly,  the largest element of $\ub$ is $M+1$ since $1$ has the unique expansion  $M^\f=MM\cdots$ in base $M+1$. Komornik and Loreti \cite{Komornik-Loreti-1998, Komornik_Loreti_2002} found the smallest element $q_{KL}=q_{KL}(M)$ of $\ub$, which was  called  the \emph{Komornik-Loreti constant} by Glendinning and Sidorov \cite{Glendinning_Sidorov_2001}. Furthermore, they showed in \cite{Komornik_Loreti_2007} that
 its topological closure $\overline{\ub}$ is a {Cantor set}: a non-empty compact set with neither isolated nor interior points. Hence,
 \begin{equation}\label{eq:11}
(1,M+1]\setminus\overline{\ub}=\bigcup(q_0,q_0^*),
\end{equation}
where  the left endpoints $q_0$ run through  $1$ and  the set $\overline{\ub}\setminus\ub$, and the right endpoints $q_0^*$ run through   a   subset $\ub^*$ of $\ub$ (cf.~\cite{DeVries_Komornik_2008}). In particular, each left endpoint $q_0$ is algebraic, while each right endpoint $q_0^*$, called a \emph{de Vries-Komornik number},  is transcendental (cf.~\cite{Kong_Li_2015}). Recently, Kalle et al.~\cite{Kalle-Kong-Li-Lv-2019} showed that the set $\ub$ has more weight  close to $M+1$. For the detailed description of the local structure of $\ub$ we refer to the recent paper \cite{Allaart-Kong-2018}.

However, for a general $x>0$ we know very little about $\ub(x)$. L\"u, Tan and Wu \cite{Lu_Tan_Wu_2014} showed that for $M=1$ and $x\in(0, 1)$ the set $\ub(x)$ is a Lebesgue null set but {has} full Hausdorff dimension. Recently, Dajani et al.~\cite{Dajani-Komornik-Kong-Li-2018} showed that the algebraic difference $\ub(x)-\ub(x)$ contains an interval for any $x\in(0, 1]$. The smallest element of $\ub(x)$ was   investigated in \cite{Kong_2016}. In this paper we will investigate the set $\ub(x)$ from the following   perspectives.  (i) We will determine the local dimension of $\ub(x)$ and establish a so-called `variation principle' in unique non-integer base expansions; (ii) We will determine the Hausdorff dimension of the symbolic set $\us(x)$ consisting of all expansions of $x$ in base $q\in\ub(x)$, and show that the function $x\mapsto\dim_H\us(x)$ is a non-increasing Devil's staircase (see Figure \ref{Figure:2}); (iii) We will determine the critical values of $\ub(x)$ {such} that when $x$ passes the first critical value  the set $\ub(x)$   changes from   positive Hausdorff dimension to a countable set, and when $x$ passes the second critical value the set $\ub(x)$ changes from an infinite set to a singleton; (iv) In contrast with $\ub=\ub(1)$ we will show that typically the set $\ub(x)$ contains isolated points.

{For $x>0$ let
\begin{equation}\label{eq:qx}
q_x:=\min\set{1+M, 1+\frac{M}{x}}.
\end{equation}
Then $q_x$   is the largest base in $(1, M+1]$ such that  the given $x$ has an expansion with respect to the alphabet $\set{0,1,\ldots, M}$, i.e., $q_x=\max\ub(x)$.}

 Our first result focuses on the local dimension of $\ub(x)$.
\begin{theorem}\label{main:local-dim}
For any $x>0$ and for any  $q\in {(1,q_x]\setminus \overline{\ub}}$ we have
\begin{align*}
\lim_{\de\to 0}\dim_H(\ub(x)\cap (q-\de, q+\de))=\lim_{\de\to 0}\dim_H(\uu_q\cap(x-\de, x+\de)).
\end{align*}
\end{theorem}
 Theorem \ref{main:local-dim} can be viewed as a `variation principle' in unique non-integer base expansions.  Recall from \cite{DeVries-Komornik-2011} the two dimensional univoque set
\[
\mathbb U=\set{(z, p): z\textrm{ has a unique }p\textrm{-expansion}}.
\]
Then the left hand side in Theorem \ref{main:local-dim} is the local dimension of the vertical slice $\mathbb U\cap\set{z=x}=\ub(x)$ at the point $(x, q)$, and the right hand side gives the local dimension of the horizontal slice $\mathbb U\cap\set{p=q}=\uu_q$ at the same point $(x, q)$. So Theorem \ref{main:local-dim} states  that for any $x>0$ and any $q\in(1, q_x]\setminus\overline{\ub}$ the local dimension  of $\mathbb U$ at the point $(x, q)$ through the vertical slice is the    same  as that through  the   horizontal slice.

Let $\set{0, 1,\ldots, M}^\N$ be the set   of all  sequences $(d_i)=d_1d_2\ldots$ over the alphabet $\set{0,1,\ldots, M}$. Equipped with the {order} topology  on $\set{0,1,\ldots,M}^\N$ induced by the metric
\begin{equation}\label{eq:metric-rho}
\rho((c_i),(d_i))=(M+1)^{-\inf\set{i\ge 1: c_i\ne d_i}}
\end{equation}
we {can define} the Hausdorff dimension of any   subset of $\set{0,1,\ldots,M}^\N$.

Note that $\us_q=\pi_q^{-1}(\uu_q)\subset \set{0,1,\ldots,M}^\N$ is the symbolic horizontal slice of the {two-dimensional} univoque set $\mathbb U$. The following result for the Hausdorff dimension of   $\us_q$ was established in \cite{Komornik-Kong-Li-17} and \cite{Allaart-Kong-2019} (see Figure \ref{Figure:1}).
\begin{figure}[h!]
  \centering
  \includegraphics[{width=12cm,height=8cm}]{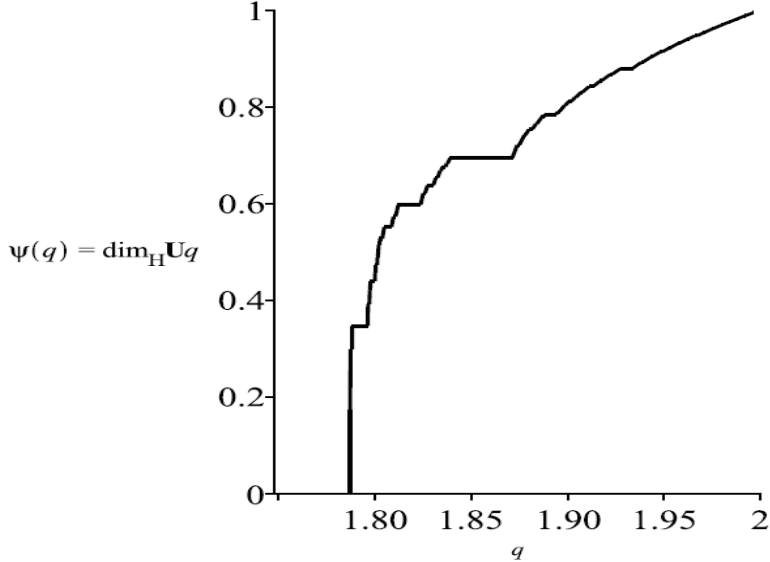}\\
  \caption{The {graph} of $\psi: q\mapsto \dim_H\us_q$ with $M=1$. $\psi(q)$ is positive if and only if $q>q_{KL}\approx 1.78723$, and $\psi(q)=1$ if and only if $q=2$.}\label{Figure:1}
\end{figure}
\begin{proposition}[\cite{Komornik-Kong-Li-17, Allaart-Kong-2019}]
  \label{th:cont-uq}
  The dimensional function $\psi: q\mapsto \dim_H\us_q$ is a non-decreasing Devil's staircase on $(1, M+1]$. {In particular,}
  \begin{itemize}
    \item   $\psi$ is   non-decreasing and continuous  on $(1, M+1]$;
    \item  $\psi$ is locally constant almost everywhere on $(1, M+1]$;
    \item $\psi(q)\in(0, 1]$ if and only if $q>q_{KL}$. Furthermore, $\psi(q)=1$ only when $q=M+1$.
  \end{itemize}
\end{proposition}

The detailed study of the plateaus of $\psi$, i.e., the largest intervals for which $\psi$ is constant,  can be found in \cite{AlcarazBarrera-Baker-Kong-2016}. For the bifurcation set of $\psi$,  which is the set of {points where $\psi$ vibrates},  we refer to \cite{Allaart-Baker-Kong-17}.

For $x>0$ let $\Phi_x(q)=x_1(q)x_2(q)\ldots$ be the quasi-greedy $q$-expansion of $x$ (see Section \ref{sec: preliminaries} for its definition). Now we define the symbolic set of univoque bases by
\[
\us(x):=\set{\Phi_x(q): q\in\ub(x)}.
\]
Observe that  for each $q\in\ub(x)$ the sequence $\Phi_x(q)\in\us(x)$ is the unique $q$-expansion of $x$.
So, the map $q\mapsto \Phi_x(q)$ is   bijective from  $\ub(x)$ to $\us(x)$. We will show in {Proposition \ref{prop:bi-holder}}    that  the map $q\mapsto\Phi_x(q)$ is locally bi-H\"older continuous {on $\ub(x)$}.

Observe that $\us(x)=\Phi_x(\ub(x))$ is the symbolic vertical slice of the two dimensional univoque set $\mathbb U$. Comparing  with {Proposition} \ref{th:cont-uq} our second main result gives the Hausdorff dimension of $\us(x)$, and shows that  the dimensional function $x\mapsto\dim_H\us(x)$ is a non-increasing Devil's staircase (see Figure \ref{Figure:2}).
\begin{figure}[h!]
  \centering
  \includegraphics[{width=12cm,height=8cm}]{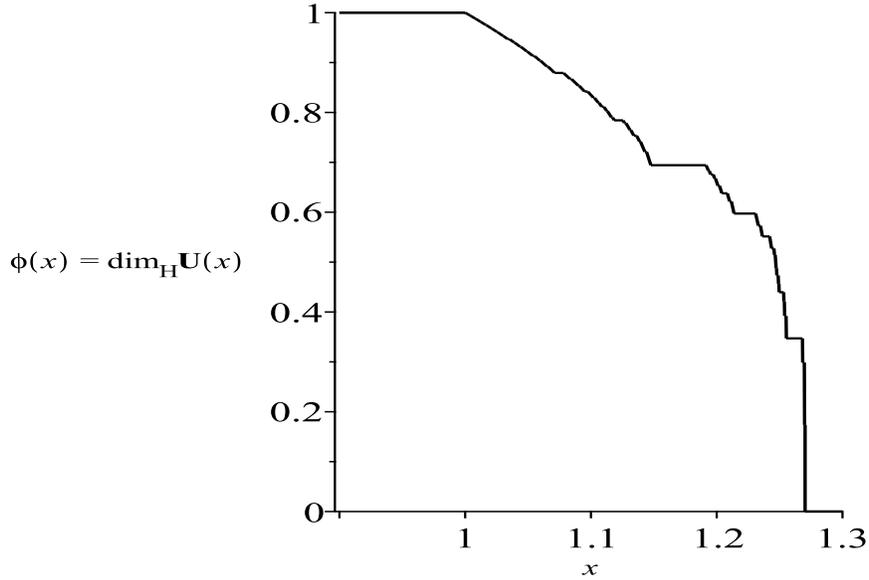}\\
  \caption{The {graph} of $\phi: x\mapsto \dim_H\us(x)$ with $M=1$. $\phi(x)$ is positive if and only if $x<x_{KL}\approx 1.27028$, and $\phi(x)=1$ if and only if $x\le 1$.}\label{Figure:2}
\end{figure}
\begin{theorem}\label{main:dim-devil}
For any $x>0$ the Hausdorff dimension of $\us(x)$ is given by
\[
\dim_H\us(x)=\dim_H\us_{q_x},
\]
where $q_x$ is defined in (\ref{eq:qx}).
Consequently, the dimensional function $\phi: x\mapsto\dim_H\us(x)$ is a non-increasing Devil's staircase on $(0,\f)$. {In particular,}
\begin{itemize}
  \item[{\rm(i)}] $\phi$ is non-increasing and continuous on $(0, \f)$;
  \item[{\rm(ii)}] $\phi$ is locally constant almost everywhere;
  \item[{\rm(iii)}] $\phi(x)\in(0, 1]$ if and only if $x<\frac{M}{q_{KL}-1}$. Furthermore, $\phi(x)=1$ if and only if $x\le 1$.
\end{itemize}
\end{theorem}

Recall from \cite{Baker_2014} that the \emph{generalized golden ratio} is defined by
\begin{equation}\label{eq:qg}
q_G=q_G(M):=\left\{\begin{array}
  {lll}
  k+1&\textrm{if}& M=2k;\\
  \frac{k+1+\sqrt{k^2+6k+5}}{2}&\textrm{if}& M=2k+1.
\end{array}\right.
\end{equation}
Note that $q_{KL}=q_{KL}(M)$ is the smallest element of $\ub=\ub(1)$ and $1<q_G<q_{KL}<M+1$. The following result on the critical values of $\uu_q=\pi_q(\us_q)$ was first proven by Glendinning and Sidorov \cite{Glendinning_Sidorov_2001} for $M=1$ and then proven in \cite{Kong_Li_Dekking_2010} for all $M\ge 2$. Furthermore, the  Hausdorff dimension of $\uu_q$ was given in \cite{Komornik-Kong-Li-17}. For a set $A$ we denote by $|A|$  its cardinality.
\begin{proposition}[\cite{Glendinning_Sidorov_2001, Kong_Li_Dekking_2010,Komornik-Kong-Li-17}]
  \label{th:critical-uq}
  For any $q\in(1, M+1]$ the Hausdorff dimension of $\uu_q$ is given by
  \[
  \dim_H\uu_q=\frac{\dim_H\us_q}{\log q}.
  \]
  Furthermore, we have the following properties.
\begin{itemize}
  \item  If $q\in(1,q_G]$, then $\uu_q=\set{0, \frac{M}{q-1}}$;
  \item If $q\in(q_G, q_{KL})$, then $|\uu_q|=\aleph_0$;
  \item If $q=q_{KL}$, then $|\uu_q|=2^{\aleph_0}$ and $\dim_H\uu_q=0$;
  \item If $q\in(q_{KL}, M+1]$, then $\dim_H\uu_q\in(0, 1]$. Furthermore, $\dim_H\uu_q=1$ if and only if $q=M+1$.
\end{itemize}
\end{proposition}
Here in {Proposition} \ref{th:critical-uq} and throughout the paper we keep using  base $M+1$ logarithms.
By  Theorem   \ref{main:dim-devil} and {Proposition} \ref{th:critical-uq} we are able to determine the critical values of $\ub(x)$ for $x>0$ and $M\ge 1$. Set
\[
x_G:=\frac{M}{q_G-1}\quad\textrm{and}\quad x_{KL}:=\frac{M}{q_{KL}-1}.
\]
Since $1<q_G<q_{KL}<M+1$, it {follows} that $1<x_{KL}<x_G$. Furthermore,  by (\ref{eq:qx}) it follows that  $q_{x_G}=q_G$ and $q_{x_{KL}}=q_{KL}$.
\begin{theorem}
  \label{main:critical-point}
  Let  $M\ge 1$.  The set $\ub(x)$ has zero Lebesgue measure for any $x>0$. {Furthermore,}
  \begin{itemize}
   \item[{\rm(i)}] {if} $x\in(0,1]$, then $\dim_H\ub(x)=1$;
   \item[{\rm(ii)}] {if} $x\in(1,x_{KL})$, then $0<\dim_H\ub(x)<1$;
   \item[{\rm(iii)}] {if} $x\in[x_{KL},x_G)$, then $|\ub(x)|=\aleph_0$;
    \item[{\rm(iv)}] {if} $x\ge x_G$, then $\ub(x)=\set{q_x}$.
  \end{itemize}
\end{theorem}
\begin{remark}
  \label{rem:critical-values}\mbox{}

  \begin{itemize}
 \item Theorem \ref{main:critical-point} (i)  was first established in \cite{Lu_Tan_Wu_2014} for $M=1$.

 \item In Lemma \ref{lem:cri-3} we present a stronger result {than} Theorem \ref{main:critical-point} (ii): for $x\in(1, x_{KL})$ we have
 \[
 0<\dim_H\uu_{q_x}\le\dim_H\ub(x)\le\max_{q\in\overline{\ub(x)}}\dim_H\uu_q<1.
 \]

 \item In contrast with {Proposition} \ref{th:critical-uq} for the univoque set $\uu_q$, Theorem \ref{main:critical-point} shows that there is no $x>0$ such that  the set $\ub(x)$ is   uncountable but has zero Hausdorff dimension.
 \end{itemize}
 \end{remark}

Recall that $\ub=\ub(1)$ has no isolated points and its closure $\overline{\ub}$ is a Cantor set. Then it is natural to ask whether this is true for $\ub(x)$?
Our {forth} main result shows that typically this is not the case.
Let
\[X_{iso}:=\set{x\in(0, \f): \ub(x)\textrm{ contains isolated points}}.\] We show that for $M=1$ the set $X_{iso}$ is dense in $(0, \f)$.
\begin{theorem}\label{main:iso-points}
 Let $M\ge 1$. The set $X_{iso}$ is dense in $[0, 1]$. {If} $M=1$, the set $X_{iso}$ is dense in $(0, \f)$.
\end{theorem}
\begin{remark}
  \label{rem:isolate-points}\mbox{}

  \begin{itemize}
    \item For $M\ge 1$ we show in Lemma \ref{lem:ux-dense-01} a slightly stronger property:  for any $x\in[0, 1] $ any neighborhood of $x$ in $X_{iso}$ contains an interval.

    \item For $M=1$ we show in Proposition \ref{prop:iso-bigger}  that  $X_{iso}\supset(1,\f)$. This means for any $x>1$ the set $\ub(x)$ contains isolated points.
  \end{itemize}
\end{remark}

 The rest of the paper is arranged in the following way. In the next section we introduce the greedy and quasi-greedy expansions, and present some useful properties {of unique expansions}. In Section \ref{sec:local structure of U(x)} we investigate the local dimension of $\ub(x)$ and prove Theorem \ref{main:local-dim}. Based on this we are able to calculate in Section \ref{sec:symbolic univoque bases} the Hausdorff dimension of the symbolic set $\us(x)$ and prove the irregularity of the dimensional function $x\mapsto\dim_H\us(x)$ (see Theorem \ref{main:dim-devil}). Furthermore, we determine the critical values of $\ub(x)$ {such} that when $x$ crosses the first critical value the Hausdorff dimension of $\ub(x)$ vanishes, and when $x$ crosses the second critical value the set $\ub(x)$ degenerates to a singleton (see Theorem \ref{main:critical-point}). The proof of Theorem \ref{main:iso-points}  is presented in Section \ref{sec:isolated points}. {Although} the set  $\ub(1)$ has no isolated points, we show that typically $\ub(x)$ contains isolated points.  In the final section we pose some remarks and questions on $\ub(x)$.

\section{Preliminaries}\label{sec: preliminaries}
{In this section we recall some well-known properties from unique non-integer base expansions. First we need some terminology from symbolic dynamics (cf.~\cite{Lind_Marcus_1995}). Let $\set{0,1,\ldots, M}^\N$ be the set of infinite sequences with digits from the alphabet $\set{0,1,\ldots, M}$. Denote by $\si$ the left shift on $\set{0,1,\ldots, M}^\N$ such that $\si((c_i))=(c_{i+1})$.
  By  a \emph{word} $\mathbf c=c_1\ldots c_n$ we mean a finite string of digits with each digit $c_i$  from   $\set{0,1,\ldots, M}$. Let $\set{0,1,\ldots, M}^*$ be the set of all  words including the empty word $\epsilon$. For two words $\mathbf c, \mathbf d\in\set{0,1,\ldots,M}^*$ we write $\mathbf c\mathbf d$ as a new word which is the concatenation of them. We denote by $\mathbf c^\f=\mathbf c\mathbf c\ldots\in\set{0,1,\ldots, M}^\N$ the periodic sequence which is the infinite concatenation of $\mathbf c$ with itself.  Throughout  the paper we will use the lexicographical ordering  ``$\prec, \lle, \succ$" or ``$\lge$" between sequences and words in the usual way. For example, for two sequences $(c_i), (d_i)\in\set{0, 1,\ldots, M}^\N$ we write $(c_i)\prec (d_i)$ if $c_1<d_1$, or there exists $n>1$ such that  $c_1\ldots c_{n-1}=d_1\ldots d_{n-1}$ and $c_n<d_n$. Furthermore, for two words $\mathbf c, \mathbf d$ we say $\mathbf c\prec \mathbf d$ if $\mathbf c 0^\f\prec \mathbf d 0^\f$. For a sequence $(c_i)$ we denote its \emph{reflection} by $\overline{(c_i)}=(M-c_1)(M-c_2)\ldots\in\set{0,1,\ldots, M}^\N$. Similarly, for a  word $\mathbf c=c_1\ldots c_n$ we denote its  {reflection} by $\overline{\mathbf c}:=(M-c_1)\ldots(M-c_n)$. If $c_n<M$, we write $\mathbf c^+:=c_1\ldots c_{n-1}(c_n+1)$; and if $c_n>0$, we write $\mathbf c^-:=c_1\ldots c_{n-1}(c_n-1)$. So, $\overline{\mathbf c}, \mathbf c^+$ and $\mathbf c^-$ are all words with each digit in $\set{0,1,\ldots, M}$.}

\subsection{Quasi-greedy and greedy expansions}
Let $M\ge 1$ and $x>0$. Recall from (\ref{eq:qx}) that $q_x=\min\set{1+M, 1+M/x}=\max\ub(x)$. For $q\in(1, q_x]$ let
\[
\Phi_x(q)=x_1(q)x_2(q)\ldots\in\set{0,1,\ldots, M}^\N
\] be
 the  \emph{quasi-greedy} $q$-expansion of $x$, which is the lexicographically largest $q$-expansion of $x$ not ending with $0^\f$.
 In other words, $\Phi_x(q)=(x_i(q))$ is the  $q$-expansion of $x$ satisfying
 \[
 \sum_{i=1}^n \frac{x_i(q)}{q^i}<x\quad\textrm{for all }n\ge 1.
 \]
In particular,  for $x=1$ and $q\in(1,q_1]=(1, M+1]$ we  reserve the notation
$
\alpha(q)=(\alpha_i(q))=\Phi_1(q)
$
for  the quasi-greedy $q$-expansion of $1$.

 Similarly, for $q\in(1, q_x]$ let
 \[\Psi_x(q)=\tilde x_1(q)\tilde x_2(q)\ldots\in\set{0,1,\ldots, M}^\N\] be the \emph{greedy} $q$-expansion of $x$, which is the lexicographically largest $q$-expansion of $x$. Then $\Psi_x(q)=(\tilde x_i(q))$ is the $q$-expansion of $x$ satisfying
\[
\sum_{i=1}^n\frac{\tilde x_i(q)}{q^i}+\frac{1}{q^n}>x\quad\textrm{whenever}\quad \tilde x_n(q)<M.
\]
If $x$ has a unique $q$-expansion, i.e., $q\in\ub(x)$, then $\Phi_x(q)=\Psi_x(q)$.

The following lemma for the quasi-greedy {expansion} $\Phi_x(q)$ and greedy {expansion} $\Psi_x(q)$ was essentially proven in \cite{DeVries-Komornik-2011}.
\begin{lemma}\label{l21}\mbox{}

\begin{enumerate}
\item[{\rm(i)}] Let $x>0$. Then the map $q\mapsto \Phi_x(q)$ is left continuous and strictly increasing  in $(1, q_x]$.
 Moreover, the sequence $\Phi_x(q)=(x_i(q))$ satisfies
\[
x_{n+1}(q)x_{n+2}(q)\cdots\lle \alpha(q)\quad\textrm{ whenever}\quad x_n(q)<M.
\]

\item[{\rm(ii)}] For $x>0$ the map $q\mapsto \Psi_x(q)$ is right continuous and strictly increasing in $(1, q_x]$. Moreover, the
sequence $\Psi_x(q)=(\tilde x_i(q))$ satisfies
\[
\tilde x_{n+1}(q)\tilde x_{n+2}(q)\cdots\prec \alpha(q)\quad\textrm{ whenever}\quad \tilde x_n(q)<M.
\]
\end{enumerate}
\end{lemma}
\begin{proof}
{The monotonicity statements in (i) and (ii) are obvious by the definitions of $\Phi_x$ and $\Psi_x$ respectively. The continuity statements follow from  \cite[Lemmas 2.3 and 2.5]{DeVries-Komornik-2011}. Finally,
the lexicographical characterizations of $\Phi_x(q)$ and $\Psi_x(q)$   can be found in \cite{Baiocchi_Komornik_2007}.}
\end{proof}
\begin{remark}
  \label{rem:chara-alpha}
Taking $x=1$ in Lemma \ref{l21} (i) it follows that the map $q\mapsto\Phi_1(q)=\al(q)$ is left-continuous and strictly increasing in $(1, M+1]$. In particular, the quasi-greedy expansion $\al(q)=(\al_i(q))$ satisfies
$
\al_{n+1}(q)\al_{n+2}(q)\ldots\lle \al(q)$ {whenever} $ \al_n(q)<M.
$
Indeed, one can verify (see also \cite[Proposition 2.3]{deVries-Komornik-Loreti-2016}) that the map $q\mapsto\al(q)$ is bijective from $(1, M+1]$ to the set of sequences $(a_i)\in\set{0,1,\ldots, M}^\N$ not ending with $0^\f$ and satisfying
\[
a_{n+1}a_{n+2}\ldots \lle a_1a_2\ldots \quad\textrm{for all }n\ge 0.
\]
\end{remark}

{\subsection{Unique expansions}
For $q\in(1, M+1]$ we recall  the symbolic univoque set
\[
\us_q=\set{(d_i)\in\set{0,1,\ldots, M}^\N:  \pi_q((d_i))\in\uu_q},
\]
where $\pi_q$ is the projection map define in (\ref{eq:pi-q}).
Then each sequence $(d_i)\in\us_q$ is the unique $q$-expansion of $\pi_q((d_i))$. So   $\pi_q$ is {a} bijective map from $\us_q$ to $\uu_q$.
The following lexicographical characterization of $\us_q$ is given by Erd\H os et al.~\cite{Erdos_Joo_Komornik_1990} (see also  \cite{DeVries_Komornik_2008}).
\begin{lemma}
  \label{lem:chara-uq}
  Let $q\in(1, M+1]$. Then $\us_q$ consists of all sequences $(d_i)\in\set{0, 1,\ldots, M}^\N$ satisfying
  \[
  \left\{
  \begin{array}
    {lll}
    d_{n+1}d_{n+2}\ldots \prec\al(q)&\textrm{if}& d_n<M,\\
    \overline{d_{n+1}d_{n+2}\ldots} \prec \al(q)&\textrm{if}& d_n>0.
  \end{array}\right.
  \]
\end{lemma}
Observe that  for any $x>0$ and $q\in(1,M+1]$ we have $q\in\ub(x)$ if and only if $\Phi_x(q)\in\us_q$.
 Recall that $\ub=\ub(1)$ is the set of bases for which $1$ has a unique expansion.  Komornik and Loreti \cite{Komornik_Loreti_2007} showed that its topological closure $\overline{\ub}$ is  a Cantor set as described in (\ref{eq:11}). Motivated by the work of de Vries and Komornik \cite{DeVries_Komornik_2008} we introduce the bifurcation set $\vb$ of the set-valued map $q\mapsto\us_q$ defined by
\begin{equation}\label{eq:v}
\vb:=\set{q\in(1, M+1]: \us_r\ne\us_q~\forall r>q}.
\end{equation}
They showed in \cite{DeVries_Komornik_2008} that $\overline{\ub}\subset\vb$, and
 $\vb\setminus\overline{\ub}$ is countably infinite.

The following intimate connection between $\us_q$, $\overline{\ub}$ {and $\vb$} was established   by de Vries and Komornik \cite{DeVries_Komornik_2008} (see also \cite{deVries-Komornik-Loreti-2016}).
\begin{lemma}
  \label{lem:prop-Uq}
  Let $M\ge 1$. The following statements hold true.
  \begin{enumerate}
    \item[{\rm(i)}] The set-valued map $q\mapsto \us_q$ is non-decreasing with respect to the set-inclusion.  Furthermore, for any connected component  $(q_0, q_0^*)$   of $(1, M+1]\setminus\overline{\ub}$ and for any $p, q\in(q_0, q_0^*)$ the difference between $\us_p$ and $\us_q$ is at most countable.
    \item[{\rm(ii)}] For each connected component $(q_L, q_R)$ of $(1,{ M+1}]\setminus\vb$ the set-valued map $q\mapsto\us_q$ is constant in $(q_L, q_R]$.
  \end{enumerate}
\end{lemma}
}

\section{Local dimension  of $\ub(x)$}\label{sec:local structure of U(x)}
In this section we will investigate the local dimension of $\ub(x)$ by showing that the map $\Phi_x$ is locally bi-H\"older continuous from $\ub(x)$ onto $\us(x)$. This provides  a good estimation for the local dimension of $\ub(x)$ via its symbolic set $\us(x)$. Based on this estimation we are able to prove the `variation principle' as described in Theorem \ref{main:local-dim}.

Recall that the symbolic set $\us(x)=\set{\Phi_x(q): q\in\ub(x)}$ and the metric $\rho$ is defined in (\ref{eq:metric-rho}). First we show that the map $\Phi_x: \ub(x)\mapsto\us(x)$ is locally bi-H\"older continuous.
\begin{proposition}
  \label{prop:bi-holder}
  Let $x>0$ and ${ 1< a < b < M+1}$. Then for any $p_1, p_2\in\ub(x) {\cap(a, b)}$,
  \begin{equation}\label{eq:inequality}
 { C_1|p_1-p_2|^{{ \frac{1}{\log a}}}\le\rho(\Phi_x(p_1), \Phi_x(p_2))\le C_2|p_1-p_2|^{{ \frac{1}{\log b}}},}
  \end{equation}
  where $C_1, C_2$ are constants independent of  $p_1$ and $p_2$.
\end{proposition}

\begin{proof}[Proof of Proposition \ref{prop:bi-holder}]
  Take $p_1, p_2\in\ub(x)\cap{(a,b)}$ with $p_1<p_2$. By Lemma \ref{l21} (i) we have $(x_i(p_1))=\Phi_x(p_1)\prec \Phi_x(p_2)=(x_i(p_2))$. Then { there exists an integer $n \ge 1$} such that
  \begin{equation}
    \label{e:k3}
    x_1(p_1)\cdots x_{n-1}(p_1)=x_1(p_2)\cdots x_{n-1}(p_2)\quad\textrm{and}\quad x_n(p_1)<x_n(p_2).
  \end{equation}
  Note that
  \[
  \sum_{i=1}^{n-1}\frac{x_i(p_1)}{p_1^i}<\sum_{i=1}^\f\frac{x_i(p_1)}{p_1^i}=x=\sum_{i=1}^\f\frac{x_i(p_2)}{p_2^i}\le\sum_{i=1}^{n-1}\frac{x_i(p_2)}{p_2^i}
  +\frac{M}{p_2^{n-1}(p_2-1)}.
  \]
  Then by (\ref{e:k3}) {it follows that}
  \[
  x(p_2-p_1)<\sum_{i=1}^{n-1}\frac{x_i(p_2)}{p_2^{i-1}}-\sum_{i=1}^{n-1}\frac{x_i(p_1)}{p_1^{i-1}}+\frac{M}{p_2^{n-2}(p_2-1)}\le\frac{M}{p_2^{n-2}(p_2-1)},
  \]
which implies
  \begin{equation}
    \label{e:k4}
    p_2-p_1<\frac{M}{p_2^{n-2}(p_2-1)x}.
  \end{equation}
{Therefore}, by (\ref{e:k3}) and (\ref{e:k4}) it follows that
  \begin{align*}
    \rho(\Phi_x(p_1), \Phi_x(p_2))^{\log a}= (M+1)^{-n\log a}= \frac{1}{a^n}
    &> \frac{1}{p_2^{n}}\\
    &> \frac{(p_2-1)x}{Mp_2^2}|p_2-p_1| \ge \frac{(a-1)x}{Mb^2}|p_2-p_1|,
  \end{align*}
This {proves the first inequality of (\ref{eq:inequality}) by taking $C_1:=(\frac{(a-1)x}{Mb^2})^{1/\log a}$}.

\medskip

{For the second inequality of (\ref{eq:inequality})
  we note that ${ b}< M+1$. So, $\al({ b})\prec \al(M+1)=M^\f$ by Lemma \ref{l21} (i)}. Then there exists $i_0\ge 1$ such that
  \begin{equation}
    \label{e:k8}
    \al_1({ b})\cdots\al_{i_0}({ b})\prec {M^{i_0}.}
  \end{equation}
   Since $p_2\in\ub(x)$, we have $\Phi_x(p_2)\in\us_{p_2}$. Then {by (\ref{e:k3}), (\ref{e:k8}) and Lemma \ref{lem:chara-uq}} it follows that
  \[
  \sum_{i=1}^n\frac{x_i(p_2)}{p_1^i}\ge\sum_{i=1}^\f\frac{x_i(p_1)}{p_1^i}=x=\sum_{i=1}^\f\frac{x_i(p_2)}{p_2^i}>\sum_{i=1}^n\frac{x_i(p_2)}{p_2^i}+\frac{1}{p_2^{n+i_0}}.
    \]
This implies
\begin{equation}\label{eq:k7'}
\frac{1}{p_2^{n+i_0}}<\sum_{i=1}^n \left(\frac{x_i(p_2)}{p_1^i}-\frac{x_i(p_2)}{p_2^i}\right)\le\sum_{i=1}^\f\left(\frac{M}{p_1^i}-\frac{M}{p_2^i}\right)=\frac{M|p_2-p_1|}{(p_1-1)(p_2-1)}.
\end{equation}
Hence,  by (\ref{e:k3}) and (\ref{eq:k7'}) it follows that
\begin{align*}
  \rho(\Phi_x(p_1), \Phi_x(p_2))^{{ \log b}}=(M+1)^{-n \log b} = \frac{1}{b^n} &< \frac{1}{p_2^n}\\
  &<\frac{M p_2^{i_0}}{(p_1-1)(p_2-1)}|p_2-p_1| \le  \frac{Mb^{i_0}}{(a-1)^2}|p_1-p_2|,
\end{align*}
 This
establishes the second inequality in (\ref{eq:inequality}) by taking $C_2:=(\frac{Mb^{i_0}}{(a-1)^2})^{1/\log b}$.
\end{proof}

The following lemma for the Hausdorff   dimension under H\"older continuous maps is well-known (cf.~\cite{Falconer_1990}).

\begin{lemma}\label{lem:32}
  Let $(X, d_1)$ and $(Y, d_2)$ be two metric spaces, and let $f:X\ra Y$. If there exist positive constants $\de, C$ and $\lambda$  such that
  $$
  d_2(f(x), f(y))\le C d_1(x, y)^\lambda
  $$
  for any $x, y\in X$ with $d_1(x, y)\le \delta$, then $\dim_H f(X)\le \frac{1}{\lambda}\dim_H X$.
\end{lemma}

By Proposition \ref{prop:bi-holder} and Lemma \ref{lem:32} we have the following {estimation} for the local dimension of $\ub(x)$, which  states that the local dimension of $\ub(x)$ at any point $q\in(1, M+1)$ can be roughly estimated by the local dimension of the symbolic set $\us(x)$ at  $\Phi_x(q)$.

{\begin{proposition}
  \label{prop:31}
  Let $x>0$ and ${1<a < b< M+1}$. Then
  \[
  {\frac{\dim_H \Phi_x\big(\ub(x)\cap(a, b)\big)}{\log b}\le\dim_H\big(\ub(x)\cap(a, b)\big)\le\frac{\dim_H \Phi_x\big(\ub(x)\cap(a, b)\big)}{\log a}. }
  \]
\end{proposition}
\begin{proof}
Excluding the trivial case we assume that $\ub(x)\cap {(a, b)}$ contains infinitely many elements.
     Note that   the map
  \begin{align*}
  \Phi_x: \ub(x)\cap {(a, b)} & \longrightarrow  {\Phi_x(\ub(x)\cap(a, b))};\qquad
  p \mapsto  \Phi_x(p)
  \end{align*}
  is bijective.
   Then its inverse map $\Phi_x^{-1}$ is well-defined. Hence, the proposition follows by Proposition \ref{prop:bi-holder} and Lemma \ref{lem:32}.
\end{proof}}

{To prove Theorem \ref{main:local-dim} we still need the following lemma.
\begin{lemma}
  \label{lem:inequality-1}
  Fix $q\in(1,M+1)$. There exist constants $C_1, C_2>0$ such that for any $\mathbf c=(c_i), \mathbf d=(d_i)\in\us_q$ we have
  \begin{equation}\label{eq:holder}
  C_1\cdot\rho(\mathbf c, \mathbf d)^{\log q}\le |\pi_q(\mathbf c)-\pi_q(\mathbf d)|\le C_2\cdot\rho(\mathbf c, \mathbf d)^{\log q}.
  \end{equation}
\end{lemma}
\begin{proof}
 Define a metric $\rho_q$ on $\us_q$ by
 \[\rho_q(\mathbf c,\mathbf d)=q^{-\inf\set{i\ge 1: c_i\ne d_i}}\]
  for any $\mathbf c,\mathbf d\in\us_q$. Then the map $\pi_q: (\us_q, \rho_q)\to(\uu_q, |\cdot|)$ is bi-Lipschitz (cf.~\cite[Lemma 2.2]{Allaart-2017}). Note that $\rho_q(\mathbf c,\mathbf d)=\rho(\mathbf c,\mathbf d)^{\log q}$. Then (\ref{eq:holder}) follows by Lemma \ref{lem:32}.
\end{proof}}

{
\begin{proof}
  [Proof of Theorem \ref{main:local-dim}]
   {Let $x>0$ and $q\in(1, q_x]\setminus\overline{\ub}$. Note by (\ref{eq:v}) that $\overline{\ub}\subset\vb$ and the difference $\vb\setminus\ub$ is countable. Then there exists a  $\de>0$ such that $q-\de>1$ and $(q-\de, q)\cap{\vb}=\emptyset$.}
   So, by Lemma \ref{l21} (i) and Lemma \ref{lem:prop-Uq}   it follows that  each $p\in\ub(x)\cap(q-\de, q)$ determines a unique $y=\pi_q(\Phi_x(p))\in\uu_q\cap(x-\eta, x)$ for some $\eta>0$ depending on $\de$. This defines a bijection
  \[
  \phi:~\ub(x)\cap(q-\de, q)\to\uu_q\cap(x-\eta, x);\quad p\mapsto\pi_q(\Phi_x(p)).
  \]
   {If the set $\ub(x)\cap(q-\de, q)$ is empty, then  so is $\uu_q\cap(x-\eta, x)$. In this case, it is trivial that
      \begin{equation}
      \label{eq:june-6-1}
      \lim_{\de\to 0}\dim_H(\ub(x)\cap(q-\de, q))=\lim_{\eta\to 0}\dim_H(\uu_q\cap(x-\eta, x)),
    \end{equation}and the limit is equal to zero.}
  In the following we assume $\ub(x)\cap(q-\de, q)\ne \emptyset$. Then we claim that $\phi$ is nearly bi-Lipschitz.

  Let  $p_1, p_2\in\ub(x)\cap(q-\de, q)$. Then by (\ref{eq:inequality}) and (\ref{eq:holder}) it follows that there exist constants $D_1, D_2>0$ such that
  \[
  {D_1|p_1-p_2|^{\frac{\log q}{\log (q-\delta)}}\le|\phi(p_1)-\phi(p_2)|\le D_2|p_1-p_2|. }
  \]
  By Lemma \ref{lem:32} this implies
  \[ {
  \dim_H(\uu_q\cap(x-\eta, x))\le\dim_H(\ub(x)\cap(q-\de, q))\le\frac{\log q}{\log (q-\delta)} \dim_H(\uu_q\cap(x-\eta, x)). }
    \]
    Letting $\de\to 0$, which implies $\eta\to 0$, we then establish (\ref{eq:june-6-1}) for any $x>0$ and $q\in(1,q_x]\setminus\overline{\ub}$.

\medskip

   {On the other hand, note that $q \in (1,q_x] \setminus\overline{\ub}$. If $q=q_x\notin\overline{\ub}$, then $q_x=1+M/x<M+1$. So, $\Psi_x(q_x)=M^\f$, and thus $x$ is the largest element of $\uu_{q_x}$. Note  that $q_x=\max\ub(x)$. Then it is clear that
   \begin{equation}
     \label{eq:june-111}
     \ub(x)\cap(q_x, q_x+\de)=\uu_{q_x}\cap(x, x+\zeta)=\emptyset\quad\textrm{for any}~\de, \zeta>0.
   \end{equation}}In the following we assume $q\in(1, q_x)\setminus\overline{\ub}$.  Choose   $\delta>0$   such that $q+\de<q_x$ and $(q, q+\de)\cap\overline{\ub}=\emptyset$. Let
\[
\Gamma_{x, q,\de}:=\set{r\in\ub(x)\cap(q, q+\de): \Psi_x(r)\in\us_q}.
\]
By Lemma \ref{l21} (ii) and Lemma \ref{lem:prop-Uq} (i) it follows that the difference between $\Gamma_{x,q,\de}$ and $\ub(x)\cap(q, q+\de)$ is at most countable. So they have the same Hausdorff dimension. Observe that each $r\in\Gamma_{x,q,\de}$ determines a unique $z=\pi_q(\Psi_x(r))\in\uu_q\cap(x,x+\zeta)$ for some $\zeta>0$ depending on $\de$. This defines a bijection
\[
\psi:\Gamma_{x,q,\de}\to\uu_q\cap(x,x+\zeta);\quad r\mapsto\pi_q(\Psi_x(r)).
\]
Hence, by (\ref{eq:inequality}) and (\ref{eq:holder}) we can prove that $\psi$ is nearly bi-Lipschitz, and then by the same argument as in the proof of (\ref{eq:june-6-1}) we   conclude that
\[
\lim_{\de\to 0}\dim_H(\ub(x)\cap(q, q+\de))=\lim_{\de\to 0}\dim_H\Gamma_{x,q,\de}=\lim_{\zeta\to 0}\dim_H(\uu_q\cap(x, x+\zeta)).
\]
This, together with (\ref{eq:june-6-1}) and (\ref{eq:june-111}), completes the proof.
\end{proof}}

\section{Hausdorff dimension and critical values of $\us(x)$}\label{sec:symbolic univoque bases}

Given $x>0$, recall that the symbolic set
$
\us(x)=\set{\Phi_x(q): q\in\ub(x)}
$
consists of all unique expansions of $x$ with bases in $\ub(x)$. Clearly, $\Phi_x$ is a bijective  map from $\ub(x)$ to $\us(x)$. Instead of looking at the set $\ub(x)$ directly we focus on the symbolic set $\us(x)$. In this section we will investigate the Hausdorff dimension of $\us(x)$ with respect to the metric $\rho$ defined in (\ref{eq:metric-rho}), and   prove Theorem \ref{main:dim-devil}. Furthermore, by using Theorem  \ref{main:dim-devil} and {Proposition} \ref{th:critical-uq} we determine the critical values of $\ub(x)$, and then prove   Theorem \ref{main:critical-point}.

\subsection{Hausdorff dimension of $\us(x)$} Our first result   states that the set-valued map $x\mapsto \us(x)$ is non-increasing on $(1,\f)$ with respect to the set inclusion.
{\begin{lemma}\label{lem:decreasing-u(x)}
{The set-valued map $x\mapsto \us(x)$ is non-increasing on $(1, \f)$. }
\end{lemma}
\begin{proof}
  Let $x\in(1, \f)$ and  $(d_i)\in\us(x)$. Then there exists a unique base $q\in\ub(x)\subseteq(1, M+1)$ such that
  \begin{equation}\label{eq:june-1}
    (d_i)=\Phi_x(q)\in\us_q.
  \end{equation}
  Take  $y\in (1,x)$. Then the equation
  \begin{equation}\label{eq:june-2}
  y=\sum_{i=1}^\f\frac{d_i}{\beta^i}
  \end{equation}
  determines a unique base $\beta\in(q, M+1)$. Observe by Lemma \ref{lem:prop-Uq} (i) that the set-valued map $q\mapsto \us_q$ is non-decreasing. Then by (\ref{eq:june-1}) it follows that $(d_i)\in\us_q\subset\us_{\beta}$. In view of (\ref{eq:june-2}) this implies that
  \[
  \Phi_y(\beta)=(d_i)\in\us_\beta.
  \]
So, $(d_i)\in\us(y)$, and thus $\us(x)\subseteq\us(y)$. This completes the proof.
\end{proof}}

Now we turn to the Hausdorff dimension of $\us(x)$.
This is based on the following lemma.
\begin{lemma}
\label{lem:subset-u-x}
Given $x\in { (0,1)}$, let $(\ep_i)=\Phi_x(M+1)$ be the quasi-greedy expansion of $x$ in base $M+1$. Then there exist a word $\mathbf w$, a {positive} integer $N$ and a strictly increasing sequence $(N_j)\subset\N$ such that
\[
\us_{N_j}(x)\subset\us(x)\quad \textrm{for all }j\ge 1,
\]
where
\[
\us_{N_j}(x):=\set{\ep_1\ldots \ep_{N+N_j} \mathbf w  d_1d_2\ldots:~ d_{n+1}\ldots d_{n+N_j}\notin\set{0^{N_j}, M^{N_j}}~\forall n\ge 0}.
\]
\end{lemma}
\begin{proof}
The proof of this lemma is similar to  \cite[Section 4]{Lu_Tan_Wu_2014}.
Let $(\ep_i)\in\set{0,1,\ldots, M}^{\N}$ be the quasi-greedy expansion of $x$ in base $M+1$. We distinguish  two cases.

{\bf(I). $(\ep_i)$ ends with $M^\f$.} Then we can write
\begin{equation}\label{eq:june-8-10}
(\ep_i)=\ep_1\ldots\ep_m\,M^\f\quad\textrm{for some }{ m\ge 1} \textrm{ with }\ep_m< M.
\end{equation}
Let $\mathbf w=\epsilon$ be the empty word,  $N=m$ and $N_j=m+j$ for $j\ge 1$. Take a sequence $(y_i)\in\us_{N_j}(x)$. Then it can be written as
\begin{equation}\label{eq:june-8-11}
(y_i)=\ep_1\ldots\ep_{N+N_j}d_1d_2\ldots=\ep_1\ldots \ep_m M^{N_j}d_1d_2\ldots,
\end{equation}
where $(d_i)\in\set{0,1,\ldots, M}^\N$  contains neither  $N_j$ consecutive $0$'s nor $N_j$ consecutive $M$'s. Let $q_j$ be the unique root in $(1, M+1)$  of the equation
\[
x=\sum_{i=1}^\f\frac{y_i}{q_j^i}.
\]
Here we emphasize that $q_j<M+1$ because $\sum_{i=1}^\f y_i/(M+1)^i<x$. We claim that $(y_i)$ is the unique $q_j$-expansion of $x$.

Observe that the tail sequence
\[
y_{m+1}y_{m+2}\ldots=M^{N_j}d_1d_2\ldots=:(\de_i)
\]
satisfies $\si^n((\de_i))\lle(\de_i)$ for all $n\ge 0$. Then by Remark \ref{rem:chara-alpha} it follows that  $(\de_i)$ is the quasi-greedy expansion of $1$ for some base $q\in(1, M+1]$, i.e., $\al(q)=(\de_i)$. {By (\ref{eq:june-8-11})} and  Lemma \ref{lem:chara-uq} it suffices to prove that $\al(q_j)\succ\al(q)=(\de_i)$. In other words, it suffices to prove
\begin{equation}
  \label{eq:case1-1}
  \sum_{i=1}^\f\frac{\de_i}{q_j^i}<1.
\end{equation}
This follows from the following calculation: By {(\ref{eq:june-8-10}) and (\ref{eq:june-8-11})} we obtain
\[
\sum_{i=1}^m\frac{\ep_i}{(M+1)^i}+\frac{1}{(M+1)^m}=\sum_{i=1}^\f\frac{\ep_i}{(M+1)^i}=x=\sum_{i=1}^\f\frac{y_i}{q_j^i}=\sum_{i=1}^m\frac{\ep_i}{q_j^i}+\frac{1}{q_j^m}\sum_{i=1}^\f\frac{\de_i}{q_j^i},
\]
which gives
\[
\sum_{i=1}^\f\frac{\de_i}{q_j^i}=q_j^m\left(\frac{1}{(M+1)^m}+\sum_{i=1}^m\Big(\frac{\ep_i}{(M+1)^i}-\frac{\ep_i}{q_j^i}\Big)\right)\le\frac{q_j^m}{(M+1)^m}<1,
\]
where the inequalities follow  by   $q_j<M+1$. This proves (\ref{eq:case1-1}).

Therefore, by Lemma \ref{lem:chara-uq} it follows that $(y_i)$ is the unique $q_j$-expansion of $x$, i.e.,  $(y_i)\in\us(x)$. Hence, $\us_{N_j}(x)\subset\us(x)$.

{\bf(II).  $(\ep_i)$ does not end with $M^\f$.}
{ Since $(\ep_i)$ is the quasi-greedy expansion of $x$ in base $M+1$,   $(\ep_i)$ does not end with $0^\f$.
Then there exists an integer $N \ge 3$ such that $\ep_{N-2}>0$. Choose $N_1>N$ such that $\ep_{N+N_1+1}>0$ and
\begin{equation}\label{eq:june-8-5}
  |\{1 \le i \le N_1: \ep_i >0 \}|  \ge N+1,\qquad |\{1 \le i \le N_1: \ep_i < M \}|  \ge N+1.
\end{equation}
In fact, we can choose a strictly increasing sequence $(N_j)$ such that $\ep_{N+N_j+1}>0$ for any $j\ge 1$. Set $\mathbf w=0M$.
Fix $j\ge 1$, and take a sequence
\begin{equation}\label{eq:june-8-2}
(y_i)=\ep_1\ldots\ep_{N+N_j}0M d_1d_2\ldots \in\us_{N_j}(x),
\end{equation}
where the tail sequence $(d_i)$ contains neither $N_j$ consecutive $0$'s nor  $N_j$ consecutive $M$'s.
It follows from (\ref{eq:june-8-5}) that the initial word $\ep_1\ldots\ep_{N+N_j}$ contains neither $N_j$ consecutive $0$'s nor $N_j$ consecutive $M$'s.} Hence, by (\ref{eq:june-8-2}) it gives that $(y_i)$ contains neither $(N_j+1)$ consecutive $0$'s nor $(N_j+1)$ consecutive $M$'s.
Note that the equation
\[
x=\pi_{q_j}((y_i))=\sum_{i=1}^\f\frac{y_i}{q_j^i}
\]
determines a unique ${q_j}\in(1,M+1)$.
Here we emphasize that $q_j<M+1$, since {$\ep_{N+N_j+1}>0=y_{N+N_j+1}$ which implies that $\sum_{i=1}^\f y_i/(M+1)^i<x$}.
Then by Lemma \ref{lem:chara-uq}, to show that $(y_i)$ is the unique $q_j$-expansion of $x$ it suffices to show that $\al(q_j)\succ M^{N_j+1}0^\f$, or equivalently, to prove
\begin{equation}\label{eq:case2-1}
\sum_{i=1}^{N_j+1}\frac{M}{q_j^i}<1.
\end{equation}

Observe {by (\ref{eq:june-8-2})} that
\[
\sum_{i=1}^{N+N_j}\frac{\ep_i}{q_j^i}<\sum_{i=1}^\f\frac{y_i}{q_j^i}=x=\sum_{i=1}^\f\frac{\ep_i}{(M+1)^i}<\sum_{i=1}^{N+N_j}\frac{\ep_i}{(M+1)^i}+\frac{1}{(M+1)^{N+N_j}}.
\]
This, combined  with ${ \ep_{N-2}> 0 } $ and $q_j<M+1$,  implies that
\[
\frac{M+1-q_j}{q_j(M+1)^{N-2}}\le \frac{1}{q_j^{N-2}}-\frac{1}{(M+1)^{N-2}}\le\sum_{i=1}^{N+N_j}\left(\frac{\ep_i}{q_j^i}-\frac{\ep_i}{(M+1)^i}\right)<\frac{1}{(M+1)^{N+N_j}}.
\]
Rearranging the above inequality yields
\[
M+1-q_j<\frac{q_j}{(M+1)^{N_j+2}}<\frac{M}{q_j^{N_j+1}},
\]
which gives $M(1-q_j^{-N_j-1})<q_j-1$. Thus,
\[
\sum_{i=1}^{N_j+1}\frac{M}{q_j^i}=\frac{M(1-q_j^{-N_j-1})}{q_j-1}<1,
\]
proving (\ref{eq:case2-1}).

Therefore, $(y_i)$ is the unique $q_j$-expansion of $x$, i.e., $(y_i)\in\us(x)$. Hence, $\us_{N_j}(x)\subset\us(x)$ for all $j\ge 1$, completing the proof.
\end{proof}

\begin{proof}[Proof of Theorem \ref{main:dim-devil}]
  Note by {Proposition} \ref{th:cont-uq} that the function $q\mapsto\dim_H\us_q$ is a non-decreasing Devil's staircase  on $(1, M+1]$. Then by the definition of $q_x$ it suffices to prove
  \begin{equation}
    \label{eq:june-11}
    \dim_H\us(x)=\dim_H\us_{q_x}\quad\textrm{for all }x>0.
  \end{equation}

  First we consider $x\in { (0,1)}$. Let $(\ep_i)=\Phi_x(M+1)$ be the quasi-greedy  expansion of $x$ in base $M+1$. Then by Lemma \ref{lem:subset-u-x} there exist a word $\mathbf w$,  a { positive} integer $N$ and a strictly increasing sequence $(N_j)\subset\N$ such that
  \begin{equation}\label{eq:june-20-1}
  \us_{N_j}(x)=\set{\ep_1\ldots \ep_{N+N_j}\mathbf w d_1d_2\ldots: (d_i)\in\Lambda_j}\subset\us(x)\quad\textrm{for all }j\ge 1,
  \end{equation}
  where
  \[
  \Lambda_j:=\set{(d_i)\in\set{0,1,\ldots, M}^\N: d_{n+1}\ldots d_{n+N_j}\notin\set{0^{N_j}, M^{N_j}}~\forall n\ge 0}.
  \]
  By (\ref{eq:june-20-1}) it follows that for any $j\ge 1$,
  \begin{equation}
    \label{eq:june-20-2}
    \dim_H\us(x)\ge\dim_H\us_{N_j}(x)=\dim_H\Lambda_j=\dim_H\us_{p_j}
  \end{equation}
  where $p_j\in(1, M+1]$ satifies
  \[
  1=\sum_{i=1}^{N_j}\frac{M}{p_j^i}.
  \]
  Note  that the  function $q\mapsto \dim_H\us_q$ is continuous.
  Letting $j\to\f$ in (\ref{eq:june-20-2}), and then {$N_j\to\f$ which implies} $p_j\to M+1$, by {Proposition} \ref{th:cont-uq} it follows that
  \[
  \dim_H\us(x)\ge  \dim_H\us_{M+1}=1.
  \]
Note that $q_x=M+1$ for all $x\in{ (0,1)}$.
Hence, $\dim_H\us(x)=1=\dim_H\us_{q_x}$ for all $x\in {(0, 1)}$. This  proves (\ref{eq:june-11}) for $x\in{ (0,1)}$.

Now we prove (\ref{eq:june-11}) for {$x\ge 1$}.   Then $q_x=1+M/x$, and  the quasi-greedy $q_x$-expansion of $x$ is $M^\f$. We claim that for any $N\in\N$ there exists an integer $J=J(N)>0$ such that
\begin{equation}
  \label{eq:june-20-3}
  \Gamma_{N,J}:=\set{M^J d_1d_2\ldots: (d_i)\in\us_{q_x, N}}\subset\us(x),
\end{equation}
where $\us_{q_x, N}$ consists of all sequences $(d_i)\in\set{0, 1,\ldots, M}^\N$ satisfying
\[
   \overline{\al_1(q_x) \ldots \al_N(q_x)} \prec d_{n+1}\ldots d_{n+N}\prec \al_1(q_x)\ldots \al_N(q_x)\quad \forall n\ge 0.
\]

This can be verified by the following observation. Take $N\in \N$.
{Since $\Phi_x(q_x)=M^\f$, by Lemma \ref{l21} we can choose $J$ sufficiently large such that
\begin{equation}\label{eq:june-20-4}
  \al_1(q_{N, J}) \ldots \al_N(q_{N,J}) = \al_1(q_x)\ldots \al_N(q_x),
\end{equation}
where $q_{N,J}$ is defined by the equation $\sum_{i=1}^{J} M q^{-i} =x$.
Note that each sequence $(y_i)\in\Gamma_{N,J}$ determines a unique base $p\in(1, q_x)$ via  the  equation
\[
\sum_{i=1}^\f\frac{y_i}{p^i}=x.
\]
Since $M^J0^\f\prec (y_i)\prec M^\f$, we must have $q_{N, J}< p < q_x$. Then  by Lemma \ref{l21} and $(\ref{eq:june-20-4})$ it follows that   $\al_1(p) \ldots \al_N(p) = \al_1(q_x)\ldots \al_N(q_x)$.}
So by Lemma \ref{lem:chara-uq} we conclude  that each $(y_i)\in\Gamma_{N,J}$ is the unique expansion of $x$ {in some base $p\in(q_{N, J}, q_x)$}. In other words, $\Gamma_{N,J}\subset\us(x)$, proving (\ref{eq:june-20-3}).

Therefore, by (\ref{eq:june-20-3})   it follows that
\begin{equation}\label{eq:june-8-3}
\dim_H\us(x)\ge\dim_H\Gamma_{N,J}=\dim_H\us_{q_x, N}\quad{\forall~ N\in\N}.
\end{equation}
{Recall from \cite[Theorem 3.1]{Allaart-Kong-2019} that $\lim_{N\to\f}\dim_H\us_{q_x,N}=\dim_H\us_{q_x}$.}
Letting $N\to \f$ in (\ref{eq:june-8-3}) we conclude   that
 \[
\dim_H\us(x)\ge  \dim_H\us_{q_x}.
\]
The reverse inequality is obvious since   $\us(x)\subset\us_{q_x}$ by Lemma \ref{lem:prop-Uq} (i).
This proves (\ref{eq:june-11}) for all $x\ge 1$.
\end{proof}

\subsection{Critical values of $\ub(x)$}
Observe by {Proposition \ref{prop:bi-holder}} that   the map $\Phi_x: \ub(x)\to\us(x)$ is bijective and locally bi-H\"older continuous.
So, to determine the critical values of $\ub(x)$ is equivalent to determine the critical values of $\us(x)$. We do this by using Theorem \ref{main:dim-devil} and {Proposition} \ref{th:critical-uq}.

Recall from (\ref{eq:qg}) that $q_G=q_G(M)\in(1, M+1)$ is the generalized {golden} ratio. Then    $x_G=M/(q_G-1)>1$. First we show that $\ub(x)$ is a singleton for $x\ge x_G$.
\begin{lemma}
  \label{lem:cri-1}
  If $x\ge x_G$, then $\ub(x)=\set{q_x}$.
\end{lemma}
\begin{proof}
  Note by {Proposition} \ref{th:critical-uq} (i) that for $q\le q_G$ the symbolic univoque set $\us_q=\set{0^\f, M^\f}$. {Since for $x\ge x_G$ we have by (\ref{eq:qx}) that $q_x\le q_G$, so}
  \[
  \us(x)\subseteq\set{0^\f, M^\f}\quad\forall ~x\ge x_G.
  \]
  If $0^\f\in\us(x)$, then $x={\pi_q (0^\f) }=0$, leading to a contradiction with our assumption that $x\ge x_G>0$. So $\us(x)=\set{M^\f}$, which implies $\ub(x)=\set{q_x}$ for all $x\ge x_G$.
\end{proof}

In the following lemma we show that $x_G$ is indeed a critical value for $\ub(x)$. Recall that $q_{KL}\in(q_G, M+1)$ is the Komornik-Loreti constant. Then $x_{KL}=M/(q_{KL}-1)\in(1, x_G)$.
\begin{lemma}\label{lem:cri-2}
For any $x< x_G$ the set  $\ub(x)$ contains infinitely many elements. In particular, for $x\in[x_{KL}, x_G)$ we have $|\ub(x)|=\aleph_0$.
\end{lemma}
\begin{proof}
  Let $x<x_G$. Then $q_x>q_G$.
  {Note by Theorem \ref{main:dim-devil} that $\dim_H \us(x)=\dim_H \us_{M+1}=1$ for $x \in (0,1]$.}
  So it suffices to prove that $\us(x)$ contains infinitely many elements for $x\in(1,x_G)$. Take $x\in(1, x_G)$. Then by (\ref{eq:qx}) it follows that  $q_x{=1+M/x}\in(q_G, M+1)$ and the quasi-greedy expansion  $\Phi_x(q_x)=M^\f$. By Lemma \ref{l21} (i) { it follows that for  $k\in\N$ sufficiently large  the equation
  \begin{equation}\label{eq:june-14}
  { \pi_{p_k}(M^k\al(q_G)) }=x
  \end{equation}
  determines a unique base $p_k\in(q_G, q_x)$, and   $p_k\nearrow q_x$ as $k\to\f$. So, there exists $K=K(x)\in\N$ such that $p_k\in(q_G, q_x)$ for any $k\ge K$. Take $k\ge K$. Then $\al(p_k)\succ\al(q_G)$. By Lemma \ref{lem:chara-uq} it follows that $M^k\al(q_G)\in\us_{p_k}$. Therefore, by (\ref{eq:june-14}) we conclude that}
    \[
  M^k\al(q_G)\in\us(x)\quad\forall k\ge K.
  \]
  This implies that $\ub(x)$ is an infinite set for any $x<x_G$.

  Observe by Lemma \ref{lem:prop-Uq} (i) that $\us(x)\subseteq\us_{q_x}$ for any $x>0$. Furthermore,  $q_x\in(q_{G}, q_{KL})$ if and only if $x\in(x_{KL}, x_G)$. By using {Proposition} \ref{th:critical-uq}  (ii) it follows that $\us(x)$ is at most countable for any $x\in(x_{KL}, x_G)$.  If $x=x_{KL}$, then $q_x=q_{KL}$ and $\Phi_x(q_x)=M^\f$. Observe that
  \begin{align*}
    \us(x)&=\set{M^\f}\cup\set{\Phi_x(p): p\in\ub(x)\cap(1, q_{KL})}\\
    &=\set{M^\f}\cup\bigcup_{n=1}^\f\set{\Phi_x(p): p\in\ub(x)\cap(1, q_{KL}-\frac{1}{2^n})}\\
    &\subseteq\set{M^\f}\cup\bigcup_{n=1}^\f\us_{q_{KL}-\frac{1}{2^n}}.
  \end{align*}
  Then  by {Proposition} \ref{th:critical-uq} (ii) we can  deduce  from the above equation  that $\us(x)$ is also a  countable {set} for $x=x_{KL}$. Therefore, $|\ub(x)|=\aleph_0$ for any $x\in[x_{KL}, x_G)$.
\end{proof}

In the next lemma we demonstrate  that $x_{KL}$ is also a critical value of $\ub(x)$.
\begin{lemma}
  \label{lem:cri-3}\mbox{}

  \begin{itemize}
  \item[{\rm(i)}] If  $x\in(0, 1]$, then $\dim_H\ub(x)=1$;
  \item[{\rm(ii)}] If $x\in(1, x_{KL})$, then
  \[
  0<\dim_H\uu_{q_x}\le\dim_H\ub(x)\le\max_{q\in\overline{\ub(x)}}\dim_H\uu_q<1.
  \]
  \end{itemize}
\end{lemma}
\begin{proof}
  (i) was first proven by L\"u, Tan and Wu \cite{Lu_Tan_Wu_2014} for $M=1$. For $M>1$ the proof was given by Xu \cite{Xu-Jiayi-2019} in his thesis.
  {For completeness we prove this by using Theorem \ref{main:dim-devil} and Proposition \ref{prop:31}.

  For $x >0$ note that $\ub(x)\subset(1, q_x]$. Then  by using the countable stability of Hausdorff dimension (cf.~\cite{Falconer_1990}) and Proposition \ref{prop:31} it follows that
  \begin{align*}
    \dim_H \ub(x)
    & = \dim_H \Big(\bigcup_n \ub(x)\cap (1+n^{-1}, q_x- n^{-1})\Big) \\
    & = \sup_n \dim_H \left(\ub(x)\cap (1+n^{-1}, q_x- n^{-1})\right) \\
    & \ge \sup_n \frac{1}{\log (q_x-n^{-1})}\dim_H \Phi_x\left(\ub(x)\cap (1+n^{-1}, q_x- n^{-1})\right) \\
    & \ge \frac{1}{\log q_x} \sup_n \dim_H \Phi_x\left(\ub(x)\cap (1+n^{-1}, q_x- n^{-1})\right) \\
    & = \frac{1}{\log q_x} \dim_H \Big(\bigcup_n \Phi_x\big(\ub(x)\cap (1+n^{-1}, q_x- n^{-1})\big)\Big) \\
    & = \frac{\dim_H \us(x)}{\log q_x}= \frac{\dim_H \us_{q_x}}{\log q_x},
  \end{align*}
  where the last equality follows by Theorem \ref{main:dim-devil}.
  Therefore, by Proposition \ref{th:critical-uq} we obtain that
  \begin{equation}\label{ineq-dim}
    \dim_H \ub(x) \ge \dim_H \uu_{q_x}\quad \forall~ x >0.
  \end{equation}
  Note by (\ref{eq:qx}) that $q_x=M+1$ for all $x\in(0,1]$. Then by (\ref{ineq-dim}) and Proposition \ref{th:critical-uq}  we conclude that $\dim_H\ub(x)=\dim_H\uu_{M+1}=1$ for all $x\in(0, 1]$.

  Now we prove (ii). Let  $x\in(1, x_{KL})$.  Then $q_x\in(q_{KL}, M+1)$.
  The first two inequalities of (ii) follows from (\ref{ineq-dim}).}
For the remaining inequalities in (ii)  we set
   \[
   \xi:=\max_{q\in\overline{\ub(x)}}\dim_H\uu_q.
   \]
   Since $\overline{\ub(x)}\subset(1, q_x]$ and $q_x\in\ub(x)\cap(q_{KL}, M+1)$, by {Proposition} \ref{th:cont-uq} it follows that $0<\xi<1$. {Take   $\ep>0$.}   By {Propositions \ref{prop:31}, Lemma \ref{lem:prop-Uq} (i) and   Proposition \ref{th:cont-uq}}  it follows that for each $q\in\overline{\ub(x)}$ there exists  $\de>0$ {such that}
  {\begin{equation}\label{eq:cri-1}
  \begin{split}
  \dim_H(\ub(x)\cap(q-\de, q+\de)) &\le \frac{\dim_H\Phi_x(\ub(x)\cap(q-\de, q+\de))}{\log(q+\de)}\\
  &\le\frac{\dim_H\us_{q+\de}}{\log (q+\delta)}=\dim_H \uu_{q+\delta}\le  \dim_H\uu_q+\ep \le  \xi+\ep.
  \end{split}
  \end{equation}}
  For each $q\in\overline{\ub(x)}$ we choose a    $\de_q\in(0, M+1-q_x)$ satisfying (\ref{eq:cri-1}). Then the collection $\set{(q-\de_q, q+\de_q): q\in\overline{\ub(x)}}$ forms an open cover  of $\overline{\ub(x)}$. Since $\overline{\ub(x)}$ is compact,   there exists a finite cover $\set{(q_i-\de_i, q_i+\de_i)}_{i=1}^N$ of $\overline{\ub(x)}$, where $\de_i:=\de_{q_i}$. By (\ref{eq:cri-1}) this implies
  \begin{align*}
  {\dim_H}\ub(x)&=\dim_H\left(\ub(x)\cap\bigcup_{i=1}^N(q_i-\de_i, q_i+\de_i)\right)\\
  &=\max_{1\le i\le N}\dim_H(\ub(x)\cap(q_i-\de_i, q_i+\de_i))\\
  &\le{  \xi+\ep}.
  \end{align*}
  Since $\ep>0$ was arbitrary, this
  completes  the proof.
\end{proof}

\begin{proof}[Proof of Theorem \ref{main:critical-point}]
By Lemmas \ref{lem:cri-1}--\ref{lem:cri-3} it suffices to prove that $\ub(x)$ has zero Lebesgue measure. This result was first proven in \cite{Lu_Tan_Wu_2014} for $M=1$ by using the Lebesgue density theorem. Here we present an alternate proof by using Proposition \ref{prop:31}.  By the same argument as in the proof of Lemma \ref{lem:cri-3} (ii) one can easily verify that for $x>0$,
\[
\dim_H\left(\ub(x)\cap(1, M+1-\frac{1}{2^n})\right)<1 \quad\textrm{for any }n\ge 1.
\]
This implies that $\ub(x)\cap(1, M+1-\frac{1}{2^n})$ has zero Lebesgue measure for all $n\ge 1$. Then we conclude that   $\ub(x)$ is a Lebesgue null set by observing
\[
\ub(x)\subseteq\set{M+1}\cup\bigcup_{n=1}^\f\left(\ub(x)\cap(1, M+1-\frac{1}{2^n})\right).
\]
\end{proof}

\section{Isolated points of $\ub(x)$}\label{sec:isolated points}
In this section we will consider the topological structure of $\ub(x)$ when $x$ varies in $(0, \f)$. In particular, we will investigate the isolated points of $\ub(x)$, and prove Theorem \ref{main:iso-points}.
Recall from (\ref{eq:11}) that $(1, M+1]\setminus\overline{\ub}=\bigcup(q_0, q_0^*)$, {and recall $\vb$ from (\ref{eq:v}). Then}  for each connected component $(q_0, q_0^*)$ {of $(1,M+1]\setminus\overline{\ub}$} we can write the elements of $\vb\cap(q_0, q_0^*)=\set{q_n}_{n=1}^\f$ in an increasing order as
\[
q_0<q_1<q_2<\cdots<q_n<q_{n+1}<\cdots,\quad\textrm{and}\quad q_n\nearrow q_0^*\textrm{ as }n\to\f.
\]
{By} Lemma \ref{lem:prop-Uq} (ii) it follows that $\us_p=\us_{q_{n+1}}$ for any $p\in(q_n, q_{n+1}]$.
For $n\ge 1$ set
\[
\us_{q_{n+1}}^*:=\us_{q_{n+1}}\setminus\us_{q_{n}}=\set{(d_i)\in\us_{q_{n+1}}: (d_i)\textrm{ ends with }\al(q_n)\textrm{ or }\overline{\al(q_n)}}.
\]
It was shown in \cite{DeVries_Komornik_2008} that $\us_{q_{n+1}}^*$ is dense in $\us_{q_{n+1}}$ for any $n\ge 1$.

First we give a sufficient condition for the set $\ub(x)$ {to include} isolated points.
\begin{proposition}\label{prop:isolate-point}
Let $(q_0, q_0^*)$ be a connected component of $(1, M+1]\setminus\overline{\ub}$, and let $\set{q_n}_{n=1}^\f=\vb\cap(q_0, q_0^*)$. Then for any
\[
x\in\bigcup_{n=1}^\f\bigcup_{p\in(q_n, q_{n+1})}\pi_p(\us_{q_{n+1}}^*)
\]
the set $\ub(x)$ contains {at least one isolated point}.
\end{proposition}
\begin{proof}
For $n\ge 1$ let $x\in\pi_p(\us_{q_{n+1}}^*)$ for some $p\in(q_n, q_{n+1})$. In the following we will show that $p$ is an isolated point of $\ub(x)$.
Note by the definition of $\vb$ that $\Phi_x(p)\in \us_{q_{n+1}}^*\subset\us_{q_{n+1}}=\us_p$. Then $\Phi_x(p)=(x_i(p))\in\us_p$. Furthermore,  by the definition of $\us_{q_{n+1}}^*$ it follows that $\Phi_x(p)$ ends with $\al(q_n)=(a_1\ldots a_m\overline{a_1\ldots a_m})^\f$ for some $m\ge 1$.   So there exists   $N\in\N$ such that
\[
\Phi_x(p) =x_1(p)\ldots x_{N}(p)(a_1\ldots a_m\overline{a_1\ldots a_m})^\f.
\]

Now suppose $p\in(q_n, q_{n+1})$ is not an isolated point of $\ub(x)$. Then by Lemma \ref{l21} (i) there exists
a $p'\in\ub(x)\cap (q_n, q_{n+1})$   such that $p'\ne p$ and  $\Phi_x(p')=(x_i(p'))$ coincides with $\Phi_x(p)$ for the first $N+2m$ digits, i.e.,
\begin{equation}\label{eq:isolated-1}
x_1(p')\ldots x_{N+2m}(p')=x_1(p)\ldots x_{N}(p)a_1\ldots a_m\overline{a_1\ldots a_m}.
\end{equation}
Observe that $\Phi_x(p')\in\us_{p'}=\us_{q_{n+1}}$ and
$\al(q_{n+1})=(a_1\ldots a_m\overline{a_1\ldots a_m}^+\,\overline{a_1\ldots a_m}a_1\ldots a_m^-)^\f.$
Then by (\ref{eq:isolated-1}) and Lemma \ref{lem:chara-uq} it follows that
\[
\Phi_x(p')=x_1(p)\ldots x_{N_1}(p)(a_1\ldots a_m\overline{a_1\ldots a_m})^\f=\Phi_x(p).
\]
This implies $p'=p$ by Lemma \ref{l21} (i), leading to a contradiction with our hypothesis.  So,  $p$ is an isolated point of $\ub(x)$.
\end{proof}

 Recall from Section \ref{sec: Introduction} that
  $X_{iso}=\set{x>0: \ub(x)\textrm{ contains isolated points}}.$
 By Lemma \ref{lem:cri-1} we see that $\ub(x)=\set{q_x}$ is  a singleton for any $x\ge x_G=M/(q_G-1)$. This implies that
  $[x_G, \f)\subset X_{iso}.$
  In the following result we show that the set $X_{iso}$ is dense in $[0, 1]$.
\begin{lemma}
  \label{lem:ux-dense-01}
  For any $x\in[0, 1]$ and any $\de>0$ the intersection $X_{iso}\cap(x-\de, x+\de)$ contains an interval.
\end{lemma}
\begin{proof}
  Take $x\in[0, 1]$ and $\de>0$. Then there exist  $y\in(x-\frac{\de}{3}, x+\frac{\de}{3})$ and an integer $N_1=N_1(x, \de)>0$ such that the quasi-greedy expansion $\Phi_y(M+1)=y_1y_2\ldots$ contains neither $N_1$ consecutive $0$'s nor $N_1$ consecutive $M$'s. By Lemmas \ref{lem:chara-uq} and  \ref{lem:prop-Uq} (i) this implies
  \begin{equation}\label{eq:sep-26-0}
  (y_i)\in\us_{q}\quad\forall q>p_{N_1},
  \end{equation}
  where $p_n$ is the root of $\sum_{i=1}^n\frac{M}{p_n^i}=1$ in $(1, M+1)$. Clearly, $p_n\nearrow M+1$ as $n\to\f$. Note that the map
  \[g:[p_{N_1}, M+1]\to\mathbb R;\quad q\mapsto\pi_q((y_i))\] is continuous, and $g(M+1)=y$. So there exists an integer  $N_2>N_1$ such that
  \begin{equation}\label{eq:sep-26-1}
  g(q)\in\left(y-\frac{\de}{3}, y+\frac{\de}{3}\right)\subseteq\left(x-\frac{2\de}{3}, x+\frac{2\de}{3}\right)\quad\forall q\in[p_{N_2}, M+1].
  \end{equation}
  Let $(q_0, q_0^*)\subset[p_{N_2}, M+1]$ be a connected component of $(1, M+1]\setminus\overline{\ub}$, and write $(q_0, q_0^*)\setminus\vb=\bigcup_{n=0}^\f(q_n, q_{n+1})$.
Take $n\ge 1$.
 Recall from  \cite[Theorem 1.4]{DeVries_Komornik_2008} that the set $\us_{q_{n+1}}^*$ is dense in $\us_{q_{n+1}}$ with respect to the metric $\rho$ defined in (\ref{eq:metric-rho}). Note by (\ref{eq:sep-26-0}) that $(y_i)\in\us_{q_{n+1}}$.  Then  there exists   a sequence $(z_i)\in\us_{q_{n+1}}^*$ such that
  \begin{equation}\label{eq:sep-26-3}
  |\pi_q((z_i))-g(q)|=|\pi_q((z_i))-\pi_q((y_i))|<\frac{\de}{3}\quad\forall q\in(q_n, q_{n+1}).
  \end{equation}
 Since $(q_n, q_{n+1})\subset[p_{N_2}, M+1]$, by (\ref{eq:sep-26-1}) and (\ref{eq:sep-26-3})  it follows that
  \[
  z^q:= \pi_q((z_i))\in(x-\de, x+\de)\quad\forall q\in(q_n, q_{n+1}).
  \]
Furthermore, by using   Proposition \ref{prop:isolate-point} we obtain  that $\ub(z^q)$ contains isolated points for any $q\in(q_n, q_{n+1})$. In other words, $X_{iso}\cap(x-\de, x+\de)$ contains the sub-interval $(z^{q_{n+1}}, z^{q_n})$.
\end{proof}

In the following we consider isolated points of $\ub(x)$ for $x>1$. When $M=1$ we show that $X_{iso}\supset(1,\f)$.
\begin{proposition}\label{prop:iso-bigger}
Let $M=1$. Then for any $x>1$ the set $\ub(x)$ contains isolated points.
\end{proposition}
Note by Lemma \ref{lem:cri-1} that $X_{iso}\supset[x_G, \f)$. {Thus} it suffices to prove that $X_{iso}$ covers $(1, x_G)$. In the following we fix $M=1$, and we will prove Proposition \ref{prop:iso-bigger} in several steps.
 Let
  $(q_0, q_0^*)=(1, q_{KL})$ be the first connected component of $(1, 2]\setminus\overline{\ub}$. Then $\vb\cap(q_0, q_0^*)=\set{q_1, q_2, q_3,\ldots}$ satisfying
\[
1=q_0<q_1<q_2<q_3<\cdots<q_{0}^*=q_{KL},\quad\textrm{and}\quad q_n\nearrow q_{KL}\textrm{ as }n\to\f.
\]
Furthermore, for each $n\ge 1$ the base $q_n\in(1, q_{KL})$ admits the  quasi-greedy expansion
\begin{equation}\label{eq:oct-8-1}
\al(q_n)=(\tau_1\ldots \tau_{2^{n}}^-)^\f,
\end{equation}
where $(\tau_i)_{i=0}^\f=01101001\ldots$ is the classical Thue-Morse sequence (cf.~\cite{Allouche_Shallit_1999}).

The following property  for the sequence $(\tau_i)$ is well-known {(see, for example, \cite{Komornik_Loreti_2002})}.
\begin{lemma}
  \label{lem:thue-morse}For any integer $n\ge 0$ we have
  \begin{itemize}
    \item[{\rm(i)}] $\tau_{2^n+1}\ldots \tau_{2^{n+1}}=\overline{\tau_1\ldots\tau_{2^n}}^+$.

    \item[{\rm(ii)}]
  $
    \overline{\tau_1\ldots\tau_{2^n-i}}\prec\tau_{i+1}\ldots \tau_{2^n}\lle \tau_1\ldots\tau_{2^n-i}\quad\forall~ 0\le i<2^n.
   $
  \end{itemize}
\end{lemma}

Now we construct sequences in $\us_{q_{n+1}}^*$.
\begin{lemma}
  \label{lem:cons-c-nk}
  For $n\ge 1$ and $k\ge 1$ let
  \[
  \c_{n,k}:=\tau_1\ldots\tau_{2^{n-1}}(\overline{\tau_1\ldots\tau_{2^{n-1}}}^+)^k(\overline{\tau_1\ldots \tau_{2^n}}^+)^\f.
  \]
  Then $\c_{n,k}\in\us_{q_{n+1}}^*$ for all $k\ge 1$.
\end{lemma}
\begin{proof}
Note by (\ref{eq:oct-8-1}) that $\c_{n,k}$ ends with $(\overline{\tau_1\ldots \tau_{2^n}^-})^\f=\overline{\al(q_n)}$. Then by Lemma \ref{lem:chara-uq} it suffices to prove
\begin{equation}\label{eq:c-nk-1}
\overline{\al(q_{n+1})}\prec \si^j(\c_{n,k})\prec \al(q_{n+1})\quad\forall j\ge 1,
\end{equation}
where $\si$ is the left-shift map.
Since  $\al(q_{n+1})$ begins with $\tau_1\ldots \tau_{2^n}$, we prove (\ref{eq:c-nk-1}) by considering the following three cases.

(I). $1\le j<2^{n-1}$. Then (\ref{eq:c-nk-1}) follows by Lemma \ref{lem:thue-morse} (ii), {which implies} that
\[
\overline{\tau_1\ldots\tau_{2^{n-1}-j}}\prec \tau_{j+1}\ldots \tau_{2^{n-1}}\lle \tau_1\ldots\tau_{2^{n-1}-j}\quad\textrm{and}\quad \overline{\tau_1\ldots \tau_j}\prec \tau_{2^{n-1}-j+1}\ldots \tau_{2^{n-1}}.
\]

(II). $2^{n-1}\le j<(k+1)2^{n-1}$. Note that $\si^{2^{n-1}}(\c_{n,k})=(\overline{\tau_1\ldots\tau_{2^{n-1}}}^+)^k(\overline{\tau_1\ldots\tau_{2^n}}^+)^\f$. Then (\ref{eq:c-nk-1}) again follows by Lemma \ref{lem:thue-morse} (ii), {which implies} that
\begin{equation}\label{eq:kong-2}
\overline{\tau_1\ldots \tau_{2^{n-1}-i}}\prec \overline{\tau_{i+1}\ldots\tau_{2^{n-1}}}^+\lle \tau_1\ldots\tau_{2^{n-1}-i}\quad\textrm{and}\quad \overline{\tau_{1}\ldots \tau_i}\prec \tau_{2^{n-1}-i+1}\ldots \tau_{2^{n-1}}
\end{equation}
for any $0\le i<2^{n-1}$.

(III). $j\ge (k+1)2^{n-1}$. Then   {(\ref{eq:c-nk-1}) follows from (\ref{eq:kong-2}) with $n-1$ replaced by $n$}.
\end{proof}

By the definition of $\c_{n,k}$ it is easy to see that
\[\c_{n,k}\nearrow \c_{n,\f}:=\tau_1\ldots \tau_{2^{n-1}}(\overline{\tau_1\ldots \tau_{2^{n-1}}}^+)^\f\quad\textrm{as } k\to\f.\]
 In the following lemma we construct   sequences in $\us_{q_{n+1}}^*$ {that}   decrease  to $\c_{n,\f}$.

\begin{lemma}
  \label{lem:cons-d-nk}
  For $n\ge 2$ and $k\ge 1$ let
  \[
  \d_{n,k}:=\tau_1\ldots\tau_{2^{n-1}}(\overline{\tau_1\ldots \tau_{2^{n-1}}}^+)^k\,\overline{\tau_1\ldots\tau_{2^{n-2}}}^+\,(\overline{\tau_1\ldots \tau_{2^{n}}}^+)^\f.
  \]
  Then $\d_{n,k}\in\us_{q_{n+1}}^*$ for all $k\ge 1$.
\end{lemma}
\begin{proof}
  It is clear that $\d_{n,k}$ ends with $\overline{\al(q_n)}$. Then by Lemma \ref{lem:chara-uq} it suffices to prove
  \begin{equation}
    \label{eq:d-nk-1}
    \overline{\al(q_{n+1})}\prec \si^j(\d_{n,k})\prec \al(q_{n+1})\quad\forall j\ge 1.
  \end{equation}
  Since $\al(q_{n+1})$ begins with $\tau_1\ldots\tau_{2^n}$, by {Cases (I) and (II) in} the proof of Lemma \ref{lem:cons-c-nk} we only need to verify (\ref{eq:d-nk-1}) for {$j\ge k 2^{n-1}+2^{n-2}$}.
  Observe by Lemma \ref{lem:thue-morse} (i) that
  \begin{align*}
  \si^{k 2^{n-1}+2^{n-2}}(\d_{n,k})&=\tau_1\ldots\tau_{2^{n-2}}\overline{\tau_1\ldots\tau_{2^{n-2}}}^+(\overline{\tau_1\ldots\tau_{2^{n-1}}}\tau_1\ldots\tau_{2^{n-1}})^\f\\
  &=(\tau_1\ldots\tau_{2^{n-1}}\overline{\tau_1\ldots\tau_{2^{n-1}}})^\f{=\al(q_n)}.
  \end{align*}
  Then {by the same argument as in the proof of Case III in Lemma \ref{lem:cons-c-nk} it follows that  (\ref{eq:d-nk-1}) holds} for all $j\ge k 2^{n-1}+2^{n-2}$. This completes the proof.
\end{proof}

Clearly,
\[\d_{n,k}\searrow \d_{n,\f}:=\tau_1\ldots \tau_{2^{n-1}}(\overline{\tau_1\ldots\tau_{2^{n-1}}}^+)^\f=\c_{n,\f}\quad\textrm{as }k\to\f.\]
 Now we are ready to prove Proposition \ref{prop:iso-bigger}.

\begin{proof}[Proof of Proposition \ref{prop:iso-bigger}]
 Let $M=1$. {By Lemma  \ref{lem:cons-c-nk} and Proposition \ref{prop:isolate-point}} it follows that
  \begin{equation}\label{eq:iso-1}
  X_{iso}\supset\bigcup_{n=1}^\f\bigcup_{k=1}^\f\bigcup_{p\in(q_n, q_{n+1})}\pi_p(\c_{n,k})=\bigcup_{n=1}^\f\bigcup_{k=1}^\f(\pi_{q_{n+1}}(\c_{n,k}), \pi_{q_n}(\c_{n,k})),
  \end{equation}
  where the bases $q_n\in\vb$ are defined as in (\ref{eq:oct-8-1}). Similarly, by Lemma \ref{lem:cons-d-nk} and Proposition \ref{prop:isolate-point} it follows that
  \begin{equation}\label{eq:iso-2}
  X_{iso}\supset\bigcup_{n=2}^\f\bigcup_{k=1}^\f\bigcup_{p\in(q_n, q_{n+1})}\pi_p(\d_{n,k})=\bigcup_{n=2}^\f\bigcup_{k=1}^\f(\pi_{q_{n+1}}(\d_{n,k}), \pi_{q_n}(\d_{n,k})).
  \end{equation}
  In the following we will show that the unions in (\ref{eq:iso-1}) and (\ref{eq:iso-2}) are sufficient to cover  $(1, x_G)$.

  First we prove that the {union} in (\ref{eq:iso-1}) covers $(1, x_G)$ up to a countable set.
{By} (\ref{eq:oct-8-1}) and Lemma \ref{lem:thue-morse} (i) it follows that
\begin{equation}
  \label{eq:oct-8-5}
  \begin{split}
    \pi_{q_{n+1}}(\mathbf c_{n,k+1})&=\pi_{q_{n+1}}(\tau_{1}\ldots\tau_{2^{n-1}}(\overline{\tau_1\ldots \tau_{2^{n-1}}}^+)^{k+1}\overline{\tau_1\ldots\tau_{2^{n-1}}}(\tau_1\ldots \tau_{2^{n-1}}\overline{\tau_1\ldots\tau_{2^{n-1}}})^\f)\\
    &<\pi_{q_{n+1}}(\tau_1\ldots\tau_{2^{n-1}}(\overline{\tau_1\ldots\tau_{2^{n-1}}}^+)^{k+2}0^\f)\\
    &=\pi_{q_{n+1}}(\tau_1\ldots\tau_{2^n} 0^{2^{n-1}k}\,\overline{\tau_1\ldots\tau_{2^{n-1}}}^+ 0^\f)+\pi_{q_{n+1}}(0^{2^n}(\overline{\tau_1\ldots\tau_{2^{n-1}}}^+)^k 0^\f).
  \end{split}
\end{equation}
On the other hand, by (\ref{eq:oct-8-1})   we obtain
\begin{equation}
  \label{eq:oct-8-6}
  \begin{split}
  \pi_{q_n}(\mathbf c_{n,k})&=\pi_{q_n}(\tau_1\ldots\tau_{2^{n-1}}\overline{\tau_1\ldots\tau_{2^{n-1}}}^+ 0^\f)+\pi_{q_n}(0^{2^n}(\overline{\tau_1\ldots\tau_{2^{n-1}}}^+)^k 0^\f)\\
  &=1+\pi_{q_n}(0^{2^n}(\overline{\tau_1\ldots\tau_{2^{n-1}}}^+)^k 0^\f).
\end{split}\end{equation}
Since
$
\pi_{q_{n+1}}(\tau_1\ldots\tau_{2^n} 0^{2^{n-1}k}\overline{\tau_1\ldots\tau_{2^{n-1}}}^+ 0^\f)<1
$
for any $n\ge 1, k\ge 1$,
  by (\ref{eq:oct-8-5}) and (\ref{eq:oct-8-6}) it follows that  $\pi_{q_{n+1}}(\mathbf c_{n,k+1})<\pi_{q_n}(\mathbf c_{n,k})$. Therefore, the intervals $J_k:=(\pi_{q_{n+1}}(\c_{n,k}),\pi_{q_n}(\c_{n,k}) )$ {with $k\ge 1$} are pairwise {overlapping}.
  So,
  \begin{equation}\label{eq:kk-1}
  \bigcup_{k=1}^\f (\pi_{q_{n+1}}(\c_{n,k}),\pi_{q_n}(\c_{n,k}) )=(\pi_{q_{n+1}}(\c_{n,1}), \pi_{q_n}(\c_{n,\f})),
  \end{equation}
  where {we recall that} $\c_{n,\f}=\tau_1\ldots\tau_{2^{n-1}}(\overline{\tau_1\ldots \tau_{2^{n-1}}}^+)^\f$. Note by Lemma \ref{lem:thue-morse} (i) that
  \begin{equation}\label{eq:kong-1}
  \c_{n,1}=\tau_1\ldots\tau_{2^{n-1}}\overline{\tau_1\ldots \tau_{2^{n-1}}}^+(\overline{\tau_1\ldots\tau_{2^n}}^+)^\f=\tau_1\ldots\tau_{2^n}(\overline{\tau_1\ldots\tau_{2^n}}^+)^\f=\c_{n+1,\f}.
  \end{equation}
 Write $z_n:=\pi_{q_n}(\c_{n,\f})$. Then by (\ref{eq:iso-1}), (\ref{eq:kk-1}) {and (\ref{eq:kong-1})} it follows that
 \[
 X_{iso}\supset\bigcup_{n=1}^\f(z_{n+1}, z_n).
 \]
 Observe that $z_1=\pi_{q_1}(\c_{1,\f})=\pi_{q_1}(1^\f)=x_G$. Furthermore, since $q_n\nearrow q_{KL}$ and $\c_{n,\f}\searrow\tau_1\tau_2\ldots$ as $n\to\f$, we have
 \[z_n=\pi_{q_n}(\c_{n,\f})\searrow  \pi_{q_{KL}}(\tau_1\tau_2\ldots)=1\quad \textrm{as }n\to\f.\]
 Therefore,
 \[
 X_{iso}\supset(1, x_G)\setminus\set{z_n: n\ge 2}.
 \]

  To complete the proof it remains to prove   $z_n\in X_{iso}$ for all $n\ge 2$. We will show that all of these points $z_n$ belong  to the {union} in (\ref{eq:iso-2}). {Recall} that for $n\ge 2$ the sequence   $\d_{n, k}$ decreases to $\d_{n, \f}=\c_{n,\f}$ as $k\to\f$. Since $\d_{n, k}$ and $\c_{n,\f}$ are both quasi-greedy $q_{n}$-expansions, by Lemma \ref{l21} {(i)} it follows that
  \begin{equation}\label{eq:iso-3}
\pi_{q_{n}}(\d_{n,k})>\pi_{q_{n}}({\c_{n,\f}})=z_{n}\quad\forall k\ge 1.
\end{equation}
On the other hand, since
\[
\lim_{k\to\f}\pi_{q_{n+1}}(\d_{n,k})=\pi_{q_{n+1}}(\c_{n,\f})<\pi_{q_{n}}(\c_{n,\f})=z_{n},
\]
by (\ref{eq:iso-3}) there exists $K\in\N$ such that for all $k\ge K$ we have
\begin{equation}\label{eq:june-8-4}
z_{n}\in(\pi_{q_{n+1}}(\d_{n,k}),\pi_{q_{n}}(\d_{n,k}))\subset X_{iso}\quad\forall n\ge 2.
\end{equation}
This completes the proof.
\end{proof}
{\begin{remark}
 By Proposition \ref{prop:isolate-point} and (\ref{eq:june-8-4}) it follows that for any
  $z_n=\pi_{q_n}(\c_{n,\f})$ with $n\ge 2$
and for any $k\ge K$ the equation
\[
(\mathbf d_{n,k})_{p_k}=z_n
\]
determines a unique $p_k\in(q_n, q_{n+1})$, which is isolated in $\ub(z_n)$. This means that for each $n\ge 2$ the
set $\ub(z_n)$ contains infinitely many isolated points.
\end{remark}}
\begin{proof}[Proof of Theorem \ref{main:iso-points}]
The theorem follows by Lemma \ref{lem:ux-dense-01} and Proposition \ref{prop:iso-bigger}.
\end{proof}

\section{Final remarks and questions}
At the end of this paper we pose some questions. In view of Theorem \ref{main:local-dim} it is natural to ask the following question.

 {\bf Question 1}: Does Theorem \ref{main:local-dim} hold for any  $x>0$ and any ${q\in\overline{\ub}}$?

 By Theorem \ref{main:dim-devil} and Theorem \ref{th:cont-uq} it follows that the bifurcation set of $\phi: x\mapsto \dim_H\us(x)$ can be easily obtained from the bifurcation set of $\psi: q\mapsto\dim_H\us_q$. More precisely, $x$ is a bifurcation point of $\phi$ if and only if $q_x$ is a bifurcation point of $\psi$. The same correspondence holds for the plateaus of $\phi$ and $\psi$. Motivated by Lemma \ref{lem:decreasing-u(x)} and the works studied in \cite{DeVries_Komornik_2008, Allaart-Baker-Kong-17} we ask the following question.

{\bf Question 2}: Can we describe the bifurcation sets of the set-valued map $x\mapsto\us(x)$? In other words, can we describe the following sets
\begin{align*}
&E_1=\set{x:\us(y)\ne\us(x)~\forall y>x}\quad\textrm{and}\quad
E_2=\set{x: \dim_H(\us(x)\setminus\us(y))>0~\forall y>x}?
\end{align*}

Although we can calculate the Hausdorff dimension of $\us(x)$ as in   Theorem \ref{main:critical-point},   we are not able to {determine} the Hausdorff dimension of $\ub(x)$ for $x\in(1, x_{KL})$.

{\bf Question 3}: What is the Hausdorff dimension of $\ub(x)$ for $x\in(1, x_{KL})$?

Finally, for the isolated points of $\ub(x)$ we have shown in Theorem \ref{main:iso-points} that $X_{iso}$ is dense in $(0,\f)$ for $M=1$. Our proof does not work for $M\ge 2$ in the interval $(1, x_G)$.

{\bf Question 4}: Is it true that $X_{iso}$ is dense in $(0,\f)$ for any $M\ge 2$? We   conjecture that
\[\ub(x) \textrm{ contains isolated points}\quad\Longleftrightarrow\quad   x\in(0,1)\cup(1,\f).\]

Up to now we know very little about the topological structure of $\ub(x)$. Clearly, for $x\ge x_G$ the set $\ub(x)=\set{q_x}$ is a singleton.

{\bf Question 5}: When is $\ub(x)$ a closed set for $x\in(0, x_G)$?

\section*{Acknowledgements}
{The authors thank the anonymous referee for many useful suggestions, especially for the simplification of the proof of Theorem \ref{main:local-dim}.}
D.~Kong thanks Pieter Allaart for providing the Maple codes of Figure \ref{Figure:1}. He was supported by   NSFC   No.~11971079 and the Fundamental
and Frontier Research Project of Chongqing No.~cstc2019jcyj-msxmX0338 {and No.~cx2019067}. W.~Li was
supported by  NSFC No.~11671147, 11571144  and Science and Technology Commission of Shanghai Municipality (STCSM)  No.~13dz2260400. F.~L\"u was supported by NSFC No.~11601358.

%

\end{document}